\documentclass[11pt]{amsart}
\usepackage[dvipsnames]{xcolor}
\usepackage{amsfonts,amssymb,amsmath,amscd,amstext}
\usepackage[colorlinks=true,linkcolor=teal,citecolor=purple,urlcolor=Plum]{hyperref}
\usepackage[utf8]{inputenc}
\usepackage{graphicx}
\usepackage{changes}
\usepackage{comment}
\usepackage[noabbrev,capitalize,nameinlink]{cleveref}
\usepackage{mathdots}
\usepackage[a4paper,top=2.3cm,bottom=2.3cm,left=1.5cm,right=1.5cm]{geometry}
\usepackage{aligned-overset}
\renewcommand{\leq}{\leqslant}
\renewcommand{\geq}{\geqslant}
\renewcommand{\le}{<}
\renewcommand{\ge}{>}

\newcommand{\hhh}{{\mathcal{H}}}
\newcommand{\hn}{\mathbb{H}^n}

\newcommand{\n}[1]{\nabla^\varepsilon_{ #1 }}

\newcommand{\vh}{\nu}

\newcommand{\rr}{{\mathbb{R}}}

\newcommand{\hh}{{\mathbb{H}}}

\newcommand{\Om}{\Omega}
\newcommand{\eps}{\varepsilon}

\newcommand{\s}{\mathcal S}

\newcommand{\A}{\mathbf{A}}
\newcommand{\B}{\mathbf{B}}
\newcommand{\W}{W}
\newcommand{\We}{W_\eps}
\newcommand{\Qe}{\mathcal Q^\eps}
\newcommand{\Qei}{\mathcal Q^{\eps_i}}
\newcommand{\Ne}{\mathcal N^\eps}

\newcommand{\Le}{\mathcal L^\eps}
\newcommand{\Nes}{\mathcal N^{\eps,\sigma}}
\newcommand{\Qes}{\mathcal Q^{\eps,\sigma}}
\newcommand{\Ae}{A^\eps}
\newcommand{\Aei}{A^{\eps_i}}
\renewcommand{\n}{\mathbf{N}}
\newcommand{\flui}{\mathcal F^{\eps_i}}

\definechangesauthor[color=red]{J}
\definecolor{champagne}{rgb}{0.97, 0.91, 0.81}

\definecolor{asparagus}{rgb}{0.53, 0.66, 0.42}
\DeclareMathOperator{\ric}{Ric}

\DeclareMathOperator{\divv}{div}

\DeclareMathOperator{\trace}{trace}

\DeclareMathOperator{\spann}{span}

\newtheorem{theorem}{Theorem}[section]
\newtheorem{proposition}[theorem]{Proposition}

\theoremstyle{definition}

\makeatletter

\theoremstyle{remark}

\numberwithin{equation}{section}

\definechangesauthor[color=purple]{S}
\definechangesauthor[color=red]{P}
\definechangesauthor[color=blue]{G}

\author[G. Giovannardi]{Gianmarco Giovannardi}
\address[G. Giovannardi]{Dipartimento di Matematica e Informatica "U. Dini", Università di Firenze  \protect\newline
\indent
Viale Morgani 67/A, 50134 Firenze (FI), Italy}
\email{gianmarco DOT giovannardi AT unifi DOT it}
\author[S.~Verzellesi]{Simone Verzellesi}
\address[S. Verzellesi]{Department of Decision Sciences and BIDSA, Bocconi University \protect\newline
\indent Via R\"ontgen 1, 20136, Milano (MI), Italy}
\email{simone DOT verzellesi AT unibocconi DOT it}

\title[The Bernstein problem in higher dimensional Heisenberg groups]{The Bernstein problem in higher dimensional Heisenberg groups}

\date{\today}

\subjclass{53C17, 53A10, 35R03}
\keywords{Minimal surfaces; Bernstein problem; Heisenberg group}
\thanks{\textit{Memberships and funding information.} The authors are members of the Istituto Nazionale di Alta Matematica (INdAM), Gruppo Nazionale per l'Analisi Matematica, la Probabilità e le loro Applicazioni (GNAMPA). 
 S.~Verzellesi received funding through INdAM-GNAMPA 2026 Project \emph{Variational, Geometric, and Analytic Perspectives on Regularity} and G. Giovannardi through INdAM-GNAMPA 2026 Project  \emph{Projects Parabolic equations with drift: well posedness, gradient estimates and boundary effects}, CUP E53C25002010001. S.~Verzellesi received funding by the European Union: the European
Research Council (ERC), through StG “ANGEVA”, project number: 101076411. Views
and opinions expressed are however those of the authors only and do not necessarily reflect
those of the European Union or the European Research Council. Neither the European Union nor the granting authority can be held responsible for them}
\begin{document}
\begin{abstract}
In every sub-Riemannian Heisenberg group $\mathbb H^n$, with $n\geq 2$, we prove the existence of a global perimeter minimizer whose boundary is a smooth, non-characteristic, entire intrinsic graph but it is not a vertical hyperplane. 
\end{abstract}
\maketitle
%\tableofcontents

%\subsection*{Statements and Declarations} The authors declare that they have no financial interests or conflicts of interest related to the subject matter or materials discussed in this paper.
\section{Introduction}
After the seminal work of Bernstein in $\rr^2$ \cite{MR1544873}, the Euclidean \emph{Bernstein problem} consists in the characterization global minimizers of the perimeter. 
Thanks to a long combined effort \cite{MR200816,MR250205,MR178385,MR157263,MR233295}, it is known that global perimeter minimizers are half-spaces provided that $n\leq 7$. On the other hand, counterexamples are available for every $n\geq8$. 
The classical approach to the problem relies on the monotonicity of the perimeter density and on the characterization of of singular minimal cones. A stronger formulation of the problem, the \emph{stable Bernstein problem}, consists in the characterization of complete, minimal, stable hypersurfaces. Combining the original contributions \cite{MR546314,MR562550,MR630142,MR0423263} with very recent results \cite{MR4871955,MR4706440,MR4546104,MR4816633,MR5094441,hong2026stablebernsteintheoremmathbbr7,mazet2024stableminimalhypersurfacesmathbb}, the stable Bernstein problem is fully solved, and its outcomes agree with those of the standard Bernstein problem. 

\medskip
The sub-Riemannian Bernstein problem substantially fits into the more general task of understanding minimal surfaces within sub-Riemannian structures \cite{MR2165405,MR2983199,MR2262784,MR2472175,MR4346009,pmc1,MR4314055,SGR,MR2043961,MR3445204,MR5029815,MR2448649,MR3276118}. This research area is particularly relevant in the prototypical model of the sub-Riemannian Heisenberg group $\hn$. 
%, which is $\rr^{2n+1}$ endowed with a suitable group law and a sub-Riemannian structure (cf. \Cref{preliminaries}),
We recall that the latter is $\mathbb R^{2n+1}$, from now on equipped with coordinates $(x,y,t)=(x_1,\ldots,x_n,y_1,\ldots,y_n,t)$, endowed with an appropriate (non-abelian) group law
realizing it as \emph{stratified Lie group}. Accordingly, its Lie algebra is generated by its \emph{horizontal layer}, and the relevant metric features of $\hh^n$ are recorded by its \emph{horizontal distribution} $\hhh$, endowed with a suitable left-invariant sub-Riemannian metric.   
Owing to the above structure, $\hh^n$ carries a natural notion of \emph{horizontal perimeter}. The sub-Riemannian Bernstein problem consists then in the characterization of global minimizers of the horizontal perimeter.  %Restricting the latter to the horizontal distribution $\hhh$, and still denoting this restriction by $\langle\cdot,\cdot\rangle$, $\hn$ inherits a \emph{sub-Riemannian structure} which realizes it as a \emph{sub-Riemannian manifold.} 
%We denote by $|\cdot|$ the norm induced by $\langle\cdot,\cdot\rangle$.
%Moreover, we denote by $\nabla$ the so-called \emph{pseudohermitian connection} (cf. e.g. \cite{MR4193432}), i.e. the unique \emph{metric connection} (cf. \cite{MR1138207}) whose torsion tensor is
%\begin{equation}\label{pseudotorsion}
%    \nabla_X Y-\nabla_Y X-[X,Y]=2\langle J(X),Y\rangle T
%\end{equation}
%for any $X,Y\in\Gamma(T\hh^n)$.

\medskip
There are two important obstruction to the application of the Euclidean strategy.
%The Heisenberg group $\hh^n$ is the Lie group $(\rr^{2n+1}, \cdot)$ endowed with the non-Euclidean product $\cdot$ defined by \eqref{eq:Hproduct}. A basis of left-invariant vector fields is given  by
%\begin{equation*}
%    X_j=\frac{\partial}{\partial x_j}+y_j\frac{\partial}{\partial t}\qquad \text{and}\qquad Y_j=\frac{\partial}{\partial x_j}-x_j\frac{\partial}{\partial t}
%\end{equation*}
%for $j=1,\ldots,n$ and $T=\frac{\partial}{\partial t}$. The horizontal distribution, which we denote by $\hhh$, is generated by $X_1,\ldots, X_n,Y_1,\ldots,Y_n$. 
The first one is the lack of an appropriate monotonicity formula \cite{MR2644313}. 
The second one is the existence of not flat-minimal cones, both smooth  \cite{MR2983199,SGR,MR2435652} and with low regularity \cite{MR2455341,MR2448649}.
%\color{blue} By the way some numerical evidences  suggests that we cannot expect the monotonicity formula (cf. \cite{DGNpreprint}).  \color{black} 
%, indeed it works only for graphical strips  \cite{MR2644313}. 

\medskip
Nevertheless, by means of \emph{ad hoc} techniques, the Bernstein problem in $\hh^1$ is well-understood. In this setting, rigidity of a (sufficiently smooth) hypersurface $S$ is nevertheless conditioned by the presence of the so-called \emph{characteristic points}, i.e. those points $p\in S$ for which the tangent space $T_pS$ coincides with the horizontal distribution $\hhh_p$. 
%The smooth non-flat minimal cones in $\hh^1$ are congruent to the hyperbolic paraboloid $t=xy$ and they contains a line of singular point.
Danielli, Garofalo, Nhieu and Pauls \cite{MR2648078} and Hurtado, Ritoré and Rosales \cite{MR2609016}) (cf. \cite{MR2165405,MR2983199,MR2435652,MR2472175}) proved that a complete, two-sided, connected, stable $C^2$-surface in $\hh^1$ is a \emph{vertical plane} (without characteristic points), a \emph{horizontal plane} (with one characteristic point) or the \emph{hyperbolic paraboloid} $t=x_1y_1$ (with a line of characteristic points). A \emph{vertical hyperplane} in $\hh^n$ is an Euclidean hyperplane which is parallel to the $t$-axis. It is identified in \cite{MR4923606} as the natural flat object, due to the vanishing of its sub-Riemannian second fundamental form.
The hyperbolic paraboloid, which is a cone for the intrinsic geometry of $\hn$, can be easily lifted to provide smooth, non-flat, minimal, characteristic cones 
\begin{equation}\label{eofjnt4rvjntrgvr22222}
    \{(x,y,t)\in\hh^n\,:\,t=\langle x, y\rangle\}.
\end{equation}
for any $n\geq 2$ \cite{MR4923606}. A relevant class of surfaces without characteristic points, from now on \emph{non-characteristic}, is that of \emph{intrinsic graphs} as introduced by Franchi, Serapioni and Serra Cassano in \cite{MR1871966}.
In this setting, Barone Adesi, Serra Cassano and Vittone \cite{MR2333095} had previously shown that entire, stable intrinsic graphs associated to a $C^2$-function are vertical planes. Stability plays a decisive role, because Danielli, Garofalo and Nhieu exhibited examples of entire, minimal, unstable intrinsic graphs \cite{MR2405158}. In the non-characteristic setting, the same conclusion was achieved by Galli and Ritoré \cite{MR3406514} in the class of non-characteristic, complete, two-sided, connected, stable $C^1$-surfaces. The $C^1$-regularity assumption was later improved to Euclidean Lipschitz regularity by Nicolussi Golo and Serra Cassano \cite{MR3984100} and by the first author and Ritoré \cite{GR24}, and to Euclidean Sobolev regularity by Nicolussi Golo, Serra Cassano and Vedovato \cite{golo2025bernsteinproblemsobolevintrinsic}. The best possible regularity to guarantee the above rigidity in $\hh^1$ is still an open problem \cite{MR4433085}.
%, while the smoothness of Lipschitz minimal intrinsic graphs in $\hh^n$ with $n\ge2$ is given by \cite{MR2774306}. For this reason we are not interested in the regularity issue.
%, notice that even in the case of intrinsic graphs the stability is necessary since the area functional is not convex such as in the Euclidean setting and there are area-stationary graphs that are not area minimizing \cite{MR2405158}. 

\medskip
Besides the above-mentioned characteristic counterexamples, very little is known in the higher dimensional case. 
In \cite{MR2333095}, the authors provide a negative answer for $n\geq 5$ in the class of entire, smooth, non-characteristic intrinsic graphs, by lifting the Euclidean graphical counterexamples available in $\rr^{n+1}$ when $n\geq 8$. However, the purely Euclidean character of these counterexamples suggested the non-optimality on the dimensional threshold $n\geq 5$. 
The only partial positive result, obtained by the authors and Pinamonti in \cite{simons}, classifies complete, minimal, stable, non-characteristic hypersurfaces in $\hh^2$. The rigidity statement of \cite{simons} relies however on three assumptions, providing some \emph{a priori} structure on both the second fundamental form $ h^\hhh$ and on the \emph{fundamental function} $\alpha$ (cf. \Cref{subsechypersurf}). Namely, rigidity holds in $\hh^2$ provided that  
\begin{align}
  h^\hhh(J(\vh),\B)&=0\quad\text{for every }\B\in\Gamma(\hhh'TS),\label{eq:P1}\tag{$\mathrm{P}1$}\\
 \B\alpha&=0\quad\text{for every }\B\in\Gamma(\hhh'TS),\label{eq:P2}\tag{$\mathrm{P}2$}\\
J(\nu)\alpha+\alpha^2&\geq 0,\label{eq:P3}\tag{$\mathrm{P}3$}
\end{align}
as well as
\begin{equation}\label{updesubintro}
     \sigma^\hhh(S\cap B_r(p))=O\left(r^{2n+1}\right)\qquad\text{as }r\to\infty\qquad\text{for some $p\in S$}.
\end{equation}
We refer to \Cref{sec_prel} for the relevant definitions. Condition \eqref{updesubintro} is natural, because it is satisfied by boundaries of perimeter minimizers. Therefore, an important open problem left by \cite{simons} consists in understanding wether rigidity holds beyond \eqref{eq:P1}, \eqref{eq:P2} and \eqref{eq:P3}, and how such properties are related to minimality.

\medskip
In this paper, we answer both questions, exhibiting smooth, non-characteristic counterexamples to rigidity, and hence providing a definitive negative answer to the higher dimensional Bernstein problem.

\begin{theorem} \label{thm_no_rigidity}
    For every $n\geq 2$, there exists $S\subseteq\hh^n$ satisfying the following properties.     
    \begin{enumerate}
        \item $S$ is a smooth, connected, embedded, two-sided, complete hypersurface without boundary;
        \item $S$ is non-characteristic;
        \item $S$ is the boundary of a global perimeter minimizer. In particular, $S$ is minimal, stable and satisfies the area growth conditions \eqref{updesubintro} for every $p\in S$.
        \item $S$ is an entire, smooth intrinsic graph along $X_1$;
        \item $S$ is not a vertical hyperplane;
        \item $S$ does not verify any of \eqref{eq:P1}, \eqref{eq:P2}, \eqref{eq:P3}.
    \end{enumerate}
\end{theorem}
The counterexamples provided by \Cref{thm_no_rigidity} can be constructed starting from the following observation: the reason why the hypersurface $S$ in \eqref{eofjnt4rvjntrgvr22222} does not provide an appropriate counterexample to rigidity is that $S$, although bounding a global perimeter minimizer,  \emph{does} have characteristic points. Therefore, the idea is to look for suitable perturbations of \eqref{eofjnt4rvjntrgvr22222} which are capable of removing its characteristic points, while of course preserving minimality. The natural choice is to consider $t$-graphs of the form 
\begin{equation}\label{eofjnt4rvjntrgvr}
   S=\{(x,y,t)\in\hh^n\,:\,t=\langle x, y\rangle+u(x)\},
\end{equation}
for a suitable function $u\in C^\infty(\rr^n)$. The fact that $u$ is independent of $x$ is helpful both in simplifying the expression of the mean curvature of $S$ and in detecting its characteristic points. Indeed, by elementary computations carried out in \Cref{sec_prel}, and denoting by $H^\hhh$ the (horizontal) mean curvature, then
\begin{equation*}
    H^\hhh=-\divv\left(\frac{Du}{\sqrt{|Du|^2+4|x|^2}}\right),\qquad  S_0=\{(x,y,t)\in S\,:\,x=0,\,Du(x)=0\}.
\end{equation*}
Therefore, since entire minimal $t$-graphs are automatically boundaries of global perimeter minimizers, we reduced the problem to finding a smooth solution, with non-vanishing gradient, to the above $n$-dimensional equation. Such solution is provided by the following result. 
\begin{theorem}\label{thm_main_1}
For every $n\geq 2$, there exists $u\in C^\infty(\rr^n)$ such that
\begin{equation}\label{eq:target}
 \divv\left(\frac{Du}{\sqrt{|Du|^2+4|x|^2}}\right)=0,
 \qquad \frac{\partial u}{\partial x_1}\geq 1\qquad\text{on }\rr^n.
\end{equation}
Moreover, $u$ is odd in $x_1$ and invariant under orthogonal transformations of $x_2,\ldots,x_n$.
%Writing $x=(x_1,z)$, the function is odd in $x_1$ and invariant under $O(n-1)$
%in the variable $z$.
\end{theorem}
We point out that the positive lower bound on the first partial derivative of $u$ not only guarantees that $S$ is non-characteristic, but ensures in addition that $S$ is an entire, smooth intrinsic graph. 

\medskip
The proof of \Cref{thm_main_1} is based on a compactness argument relying on a well-established Riemannian approximation technique, and can be summarized in the following four steps.
\begin{enumerate}
    \item First, we solve the Dirichlet problem associated with an appropriate regularization of \eqref{eq:target}, namely
\begin{equation}\label{rtfojtrgvojtgvntrjgn}
 \divv\left(\frac{Du}{\sqrt{|Du|^2+4|x|^2+\eps^2}}\right)=0\quad\text{on }B_R,
 \qquad u=Kx_1\quad\text{on }\partial B_R.
\end{equation}
Here $B_R$ is the standard Euclidean ball, $\eps\in(0,1]$ is the regularization parameter, and $K$ is a positive constant. The choice of the boundary contrition is tailored to provide the relevant positive lower bound on the first partial derivative.
\item Then, we provide local $C^1$-estimates for solutions to \eqref{rtfojtrgvojtgvntrjgn} which are independent of $R\ge 0$, $\eps\in(0,1]$. Precisely, height estimates are established by exhibiting suitable uniform barriers,  which are available provided that a precise choice of $K$ is made. Instead, gradient estimates follows combining a gradient maximum principle with interior gradient estimates, in the style of \cite{MR5029815}, away from the origin. 
\item A careful (but elementary) analysis of the linearized operator associated with \eqref{rtfojtrgvojtgvntrjgn} allows to provide a uniform positive lower bound for the first partial derivative of solutions to \eqref{rtfojtrgvojtgvntrjgn}.
\item Finally, we let $R\to\infty$ and $\eps\to 0$. By a compactness argument, we find a locally Lipschitz continuous limit of solutions to \eqref{rtfojtrgvojtgvntrjgn}, say $u$. By standard regularity theory for elliptic equations, the regularity of $u$ is improved outside the origin. Accordingly, $u$ inherits from the approximating sequence the positive lower bound on its first partial derivative. In turn, such property allows to exploit the regularity results of \cite{MR2774306} to show that $u$ is smooth at the origin, and eventually that it solves \eqref{eq:target}.
\end{enumerate}
We point out that solutions to \eqref{eq:target} are not available when $n=1$. Indeed, it is easy to see that 
 \[
    \left(\frac{u'}{\sqrt{4x^2+(u')^2}}\right)'=0\qquad \text{in }\rr
    \]
   does not admit entire solutions of class $C^2$ and such that $u'\geq 1$. Because of the positive answers to the Bernstein problem available in $\hh^1$, this observation is therefore not surprising. 
   \subsection*{Acknowledgments}
The authors are deeply grateful to L. Capogna, G. Citti, M. Fogagnolo, R. Monti, A. Pinamonti,
J. Pozuelo, M. Ritoré, F. Serra Cassano, and D. Vittone for many years
of stimulating discussions about this problem and for their constant
interest, encouragement, and support.

\subsection*{AI usage statement} 
The authors acknowledge the use ChatGPT-6 Astra, in the following way. 
Initially, the authors asked ChatGPT-6 Astra to look for relations occurring between minimality and \eqref{eq:P1}. Precisely, the authors gave \cite{simons} as input, asking ChatGPT-6 Astra to focus on \cite[Example 3.8]{simons}, namely \eqref{eofjnt4rvjntrgvr22222} of this paper, as an instance of characteristic minimal hypersurface satisfying \eqref{eq:P1}, and asked it try to understand how a perturbation as in \eqref{eofjnt4rvjntrgvr} would have affected either minimality or \eqref{eq:P1}. After a long series of interactions, ChatGPT-6 Astra noticed that a solution to \eqref{eq:target}, because of its derivative bound, would have produced a minimal non-characteristic hypersurface which would have not satisfied \eqref{eq:P1}. Incidentally, the authors realized that such a construction would have provided a counterexample to the higher dimensional Bernstein problem. Subsequently, the authors implemented the Riemannian approximation scheme. To this aim, the contribution of ChatGPT-6 Astra have been clearly crucial in detecting the barriers in \Cref{erfnjtrvijng}, whose computations have been additionally double-checked with the aid of Claude. Moreover, ChatGPT-6 Astra helped in sketching the monotonicity argument in \Cref{fvrtg5ytbgtygvtr}. 
The paper has been entirely written and checked by the authors, who takes full
responsibility for its content.
%\begin{remark}
  %  Notice that unique $C^2$ solution $u:\rr \to \rr$ of the problem 
  % 
  %  is $u=c$ where $c$ in $\rr$.
  %  
  %  Indeed we have that there exist $k$ in $\rr$ such that 
  %  \[
  %  u'(x)=k \sqrt{4x^2+(u')^2}
   %% \]
  %  Then we get that $0<|k|<1$ and 
  %  \[
  %  u'(x)=4\frac{|k|}{\sqrt{1-|k|^2}} |x|
  %  \]
 %   Then setting $a= 4\frac{|k|}{\sqrt{1-|k|^2}}$ we get 
 %   \[
 %   u(x)=ax|x|+ c, \qquad a,b \in \rr.
   % \]
 %   If we require $u$ to be $C^2$ then %$a=0$ and all the solutions are $u(x)=c$ for $c \in \rr$.
%\end{remark}
\section{Preliminaries}\label{sec_prel}
\subsection{Heisenberg groups}\label{subsec_hg} Fix $n\geq 1$. The $n$-th Heisenberg group $(\hn,\cdot)$ is $\mathbb R^{2n+1}$ with the group law
\begin{equation}
\label{eq:Hproduct}
    p\cdot p'=( x, y,t)\cdot (x',y',t')=\left( x+ x', y+ y', t+t'+\sum_{j=1}^n\left(x_j'y_j-x_jy_j'\right)\right),
\end{equation}
where $p=(x,y,t)=(x_1,\ldots,x_n,y_1,\ldots,y_n,t)$.
This group law realizes $\hn$ as stratified Lie group. Its \emph{horizontal distribution} $\hhh$ is generated by the left-invariant vector fields
\begin{equation*}
    Z_i=X_i=\frac{\partial}{\partial x_i}+y_i\frac{\partial}{\partial t},\qquad Z_{n+i}=Y_i=\frac{\partial}{\partial y_i}-x_i\frac{\partial}{\partial t},\qquad i=1,\ldots,n.
\end{equation*}
 A vector field which is tangent to $\hhh$ at every point is called \emph{horizontal}.
According to the above definitions, $(x,y,t)$ are \emph{exponential coordinates of the first kind}, namely
\begin{equation*}
    (x_1,\ldots,x_n,y_1,\ldots,y_n,t)=\mathrm{exp}\left(\sum_{i=1}^nx_i X_i+\sum_{i=1}^n y_i Y_i+tT\right).
\end{equation*}
%Accordingly, the \emph{horizontal gradient} of a sufficiently regular function is the horizontal vector field defined by 
 %\begin{equation*}
 %    \nabla^\hhh f=\sum_{i=1}^{2n}Z_ifZ_i.
% \end{equation*}
 Setting $Z_{2n+1}=T=\frac{\partial}{\partial t}$, then $
Z_1,\ldots,Z_{2n+1}
$ is a global frame of left-invariant vector fields. 
The only nontrivial commutation relations among $Z_1,\ldots,Z_{2n+1}$ are
\begin{equation*}
    [X_i,Y_i]=-2T,\qquad i=1,\ldots,n.
\end{equation*}
The \emph{complex structure} $J:\Gamma(T\hh^n)\to\Gamma(T\hh^n)$ is the unique $C^\infty(\hh^n)$-linear map which satisfies
\begin{equation*}
    J(X_i)=Y_i,\qquad J(Y_i)=-X_i\qquad\text{and}\qquad J(T)=0,\qquad i=1,\ldots,n.
\end{equation*}
 The triple $(\hh^n,\hhh,J)$ is a prototypical \emph{pseudohermitian manifold} (cf. \cite[Appendix]{MR2165405}). 
For any $\eps>0$, we denote by $\left\langle\cdot,\cdot\right\rangle_\eps$ the unique Riemannian metric on $\hh^n$ such that $(X_1,\ldots,X_n,Y_1,\ldots,Y_n,\eps T)$ is a global orthonormal frame.  Denote by $\ \ric^\eps $ the induced Ricci tensor. We may let $\langle\cdot,\cdot\rangle_1=\langle\cdot,\cdot\rangle$. Restricting any of the above metrics to $\hhh$ endows $\hn$ with a sub-Riemannian structure $\left(\hh^n,\hhh,\langle\cdot,\cdot\rangle|_\hhh\right)$. 
%We denote the Riemannian gradient by 
%\begin{equation*}
%    \nabla f=\sum_{i=1}^{2n+1}Z_i f Z_i.
%\end{equation*}
 The \emph{pseudohermitian connection} $\nabla$ (cf. \cite{MR4193432}) is the unique metric connection with torsion 
\begin{equation*}\label{pseudotorsion}
        \nabla_\A\B-\nabla_\B \A-[\A,\B]=2\langle J(\A),\B\rangle T,\qquad \A,\B\in\Gamma(T\hh^n).
\end{equation*}
% We recall that $\nabla$ vanishes along left-invariant vector fields (cf. \cite{MR2898770}), i.e.
%\begin{equation*}\label{phflat}
%    \nabla Z_i=0,\qquad i=1,\ldots,2n+1.
%\end{equation*}
In addition, $\hh^n$ carries a \emph{homogeneous structure} provided by \emph{intrinsic dilations} (cf. ~\cite{MR2363343}). 
Namely, we set 
\begin{equation*}
\delta_\lambda(x,y,t) = (\lambda x,\lambda y,\lambda^2 t) \qquad \text{for every $\lambda \geq 0$, $(x,y,t) \in\hh^n$}.
\end{equation*}
In this way, $\delta_\lambda$ is a Lie group isomorphism of $\hh^n$ for any $\lambda>0$. The Riemannian volume induced by $\langle\cdot,\cdot\rangle$ is the Haar measure of the group, i.e. the standard Lebesgue measure.
It satisfies the homogeneity condition
\begin{equation*} \label{eq_dilation_Lebesgue}
|\delta_\lambda(E)|=\lambda^Q|E| \qquad \text{for every $E \subseteq\hh^n$ measurable, $\lambda \geq 0$,}
\end{equation*}
where $Q\coloneqq 2n+2$ is known as \emph{homogeneous dimension} of $(\hh^n,\cdot)$ (cf. \cite{MR3587666}). 
%Therefore, the Riemannian divergence induced by $\left\langle\cdot,\cdot\right\rangle$ is the Euclidean divergence, and can be computed by 
%\begin{equation}\label{divergenza4567}
 %   \divv\A=\divv\left(\sum_{i=1}^{2n+1}A^iZ_i\right)=\sum_{i=1}^{2n+1}Z_iA^i,\qquad \A\in\Gamma(T\hh^n).
%\end{equation}
\subsection{Perimeter and perimeter minimizers} Let $\Om\subseteq\hh^n$ be an open set. If $E\subseteq \hn$ is measurable, its \emph{(horizontal) perimeter} in $\Om$ \cite{MR1871966, MR1984849,MR1404326}) is defined by
\begin{equation*}
    P_{\mathbb H}(E,\Om):=\sup\left\{\int_E\divv(\A)\,d\mathcal{L}^{2n+1}\,:\,\A\in \Gamma_c(\Om,\hhh),\,\langle\A,\A\rangle\leq 1\right\},
\end{equation*}
where $\Gamma_c(\Om,\mathcal{H})$ is the class of horizontal vector fields which are compactly supported in $\Om$. We say that $E$ has locally finite perimeter if  $P_{\mathbb H}(E,\Om)<+\infty$ whenever $\Om$ is bounded. Finally, we recall \cite{MR3587666} that an locally finite perimeter set $E$ is an \emph{global perimeter minimizer} whenever
\begin{equation*}
    P_{\mathbb H}(E,\Om)\leq P_{\mathbb H}(F,\Om)
\end{equation*}
for any $\Om \Subset\hn$ and for any $\mathbb H$-Caccioppoli set $F$ such that $E\Delta F\Subset\Om$. 
\subsection{Hypersurfaces}\label{subsechypersurf}  Let $S\subseteq\hh^n$ be a smooth, connected, embedded, two-sided hypersurface without boundary. Recall that $p\in S$ is called a \emph{characteristic point} if $\hhh_p=T_p S$. The set of characteristic points of $S$ is denoted by $S_0$. If $S_0=\emptyset$, we say that $S$ is non-characteristic. At non-characteristic points, the \emph{horizontal tangent space} $\hhh TS$ is the smooth, $(2n-1)$-dimensional distribution defined by 
\begin{equation*}
    \hhh T_p S=\hhh_p\cap T_p S,\qquad p\in S\setminus S_0.
\end{equation*}
When $n\geq 2$ and $S$ is non-characteristic, \cite{MR4103357} grants that the notion of completeness can be equivalently given by the restriction of the ambient metric on $S$ or by the intrinsic metric of $S$, provided that the latter are induced by the Euclidean, Riemannian or sub-Riemannian structure of $\hh^n$. Therefore, we can talk about completeness without ambiguity.
Denote by $\n$ the Riemannian unit normal to $S$, and by $\n^\hhh$ its orthogonal projection onto $\hhh$. 
Then, the \emph{horizontal unit normal }  $$\vh=\frac{\n^\hhh}{|\n^\hhh|}$$ 
is well-defined on $S\setminus S_0$, and is the unique, up to sign, horizontal unit vector field orthogonal to $\hhh TS$. Notice that $p$ is a characteristic point if and only if $|\n^\hhh|=0$.
%Close to every non-characteristic point, it is always possible to extend $\vh$ to a full neighborhood in $\hh^n$ by setting 
  % \begin{equation}\label{normcondist}
  %  \vh=\nabla^\hhh d,
%\end{equation}
 %  where $d$ is the signed \emph{Carnot-Carathéodory distance} from $S$ (cf. \cite{simons,MR4193432}). %Indeed, locally near $p$, $d$ is smooth and satisfies the eikonal equation $|\nabla^\hhh d|=1$ . 
 %  Henceforth, $\vh$ is always extended as in \eqref{normcondist}.
   The \emph{fundamental function} $\alpha$ is defined on $S\setminus S_0$ as the unique smooth function such that 
   \begin{equation*}\label{firstdefinitionofthevecotrs}
       \s\coloneqq T-\alpha\vh\in\Gamma(TS).
   \end{equation*}
  % It is known (cf. \cite{simons,MR4923606}) that $\alpha=Td$. Moreover, if $n=1$, then \cite[Proposition 3.1]{simons} yields
  % \begin{equation}\label{nablaessevh}
   %    \nabla_\s\vh=(J(\vh)\alpha+2\alpha^2)J(\vh).
 %  \end{equation}
   Denote by $\hhh' TS$ the distribution defined by
\begin{equation*}
    \hhh'T_p S=\hhh T_pS\cap J\left(\hhh T_p S\right),\qquad p\in S\setminus S_0.
\end{equation*}
Then, $\hhh'TS$ is a $(2n-2)$-dimensional sub-bundle of $\hhh TS$, and the latter can be orthogonally decomposed as $
    \hhh TS=\hhh' TS\oplus\spann J(\vh). $
In the first Heisenberg group, $\hhh' TS=\{0\}$. 
It is easy to check that 
\begin{equation*}
    \n=\frac{1}{\sqrt{1+\alpha^2}}\vh+\frac{\alpha}{\sqrt{1+\alpha^2}}T,\qquad \alpha=\frac{\left\langle \n,T\right\rangle}{|\n^\hhh|}.
\end{equation*}
%\begin{equation}
  %  TS=\hhh' TS\oplus\spann J(\vh)\oplus\spann \s,
%\end{equation}
%where 
%\begin{equation}\label{formadiesseepsilon}
 %   \s=-\frac{\alpha}{\sqrt{1+\eps^2\alpha^2}}\vh+\frac{1}{\sqrt{1+\eps^2\alpha^2}}\eps T.
%\end{equation}
%Moreover, set
%\begin{equation*}
%    \s\coloneqq\lim_{\eps\to 0}\frac{1}{\eps}\s^\eps=-\alpha\vh+T.
%\end{equation*}
%We recall that the \emph{pseudohermitian connection} of $\hh^n$, say $\nabla$, is the unique metric connection such that
%\begin{equation*}
%    \nabla_{Z_i}Z_j=0,\qquad i,j=1,\ldots,2n+1.
%\end{equation*}
%Accordingly, if $\A,\B,\C\in\Gamma(\hhh)$, \eqref{levi_civita} implies that
%\begin{equation}\label{relazionetraleconnessioni}
%    \left\langle\nabla^\eps_\A\B,\C\right\rangle_\eps=\left\langle\nabla_\A\B,\C\right\rangle.
%\end{equation}
The \emph{horizontal second fundamental form} $h^\hhh$ and the \emph{symmetric horizontal second fundamental form} (cf. \cite{simons,MR4193432}) $\tilde h^\hhh$ are defined on $S\setminus S_0$ respectively by
\begin{equation*}
    h^\hhh(\A,\B)\coloneqq\left\langle A^\hhh(\A),\B\right\rangle,\qquad \tilde h^\hhh(\A,\B)\coloneqq\frac{h^\hhh(\A,\B)+h^\hhh(\B,\A)}{2},\qquad\A,\B\in\Gamma(\hhh TS),
\end{equation*}
 where $ A^\hhh(\A)\coloneqq \nabla_\A\vh$ is the \emph{horizontal shape operator}. The forms $h^\hhh$ and $\tilde h^\hhh$ are related (cf. \cite{simons}) by
\begin{equation*}\label{rapptrahetildacca2026}
    \tilde h^\hhh(\A,\B)=h^\hhh(\A,\B)+\alpha\left\langle J(\A),\B\right\rangle,\qquad\A,\B\in\Gamma(\hhh TS).
\end{equation*}
%For any $\eps>0$, denote by $\sigma^\eps$ the Riemannian surface measure induced by $g_\eps$. It is easy to check that 
%\begin{equation}\label{areaelementeps}
 %   \sigma^\eps=\frac{\sqrt{1+\eps^2\alpha^2}}{\eps\sqrt{1+\alpha^2}}\sigma^1.
%\end{equation}
The \emph{horizontal mean curvature} is then defined on $S\setminus S_0$ by
\begin{equation*}
    H^\hhh=\trace h^\hhh=\trace \tilde h^\hhh.
\end{equation*}
We say that $S$ is \emph{minimal} if $H^\hhh=0$ on $S\setminus S_0$. 
Finally, the relevant sub-Riemannian surface measure $\sigma^\hhh$ is defined (cf. \cite{MR2354992, MR2262196}) by
\begin{equation}\label{areaelementorizz}
    \sigma^\hhh\coloneqq \frac{1}{\sqrt{1+\alpha^2}}\sigma,
\end{equation}
where $\sigma$ is the Riemannian surface measure induced by $\langle\cdot,\cdot\rangle$. 
%The \emph{tangent pseudohermitian connection} $\nabla^S$ is the affine connection defined on $S$ by
%\begin{equation*}
 %   \nabla^S _\A \B=\nabla_\A\B-\langle\nabla_\A\B,\vh\rangle\vh,\qquad \A,\B\in\Gamma(\hhh TS).
%\end{equation*}
%Let $\E_1,\ldots,\E_{2n-1}$ be any local orthonormal frame of $\hhh TS$. If $\varphi\in C^\infty(S\setminus S_0)$ and $\A\in\Gamma(\hhh TS)$ is supported in $S\setminus S_0$, the \emph{horizontal tangential gradient} of $\varphi$ and the  \emph{horizontal tangential divergence} of $\A$ are defined by 
%\begin{equation*}
%    \nabla ^{\hhh,S}\varphi=\sum_{i=1}^{2n-1}\left(\E_i\varphi\right)\E_i,\qquad \divv^{\hhh,S}\A\coloneqq\sum_{i=1}^{2n-1}\left\langle \nabla^S_{\E_i}\A,\E_i\right\rangle.
%\end{equation*}
%The \emph{horizontal tangential Laplacian} and the \emph{modified horizontal tangential Laplacian} are then defined by
%\begin{equation*}
 %   \Delta^{\hhh,S} \varphi=\divv^{\hhh,S}\nabla^{\hhh,S}\varphi,\qquad \hat\Delta^{\hhh,S}\varphi\coloneqq \Delta^{\hhh,S}\varphi +2\alpha J(\vh)\varphi.
%\end{equation*}
 %   Unlike $\Delta^{\hhh,S}$  (cf. \cite{MR2354992}), $\hat\Delta^{\hhh,S}$ is self-adjoint on $C^\infty_c(S\setminus S_0)$.
 In the following, if $S$ is a hypersurface as in \Cref{subsechypersurf}, we assign an upper index $\eps$ to the geometric quantities associated with $S$ and induced by $\left\langle\cdot,\cdot\right\rangle_\eps$. In particular, we denote by $\n^\eps$ its unit normal, by $h^\eps$ its second fundamental form, by $H^\eps$ its mean curvature, by $\nabla^{\eps,S}$ its gradient and by $\Delta^{\eps,S}$ its Laplace-Beltrami operator. In particular, we recall \cite{formuledivariazioneriem} that 
\begin{equation}\label{normaleriemannianaepsilon}
    \n^\eps=\frac{1}{\sqrt{1+\eps^2\alpha^2}}\vh+\frac{\eps\alpha}{\sqrt{1+\eps^2\alpha^2}}\eps T.
\end{equation}
   By \eqref{normaleriemannianaepsilon}, the Jacobi-type identity satisfied by the vertical component of the normal of a minimal (with respect to $\langle\cdot,\cdot\rangle_\eps$) hypersurface, provided by \cite[(1.5)]{MR5029815}, reads as 
    \begin{equation}\label{jacobitipeequation384969}
        -\Delta^{\eps,S}\left(\frac{\alpha}{\sqrt{1+\eps^2\alpha^2}}\right)=\Big(\left|h^\eps\right|^2+\ric^\eps(\n^\eps,\n^\eps)\Big)\left(\frac{\alpha}{\sqrt{1+\eps^2\alpha^2}}\right).
    \end{equation}
    According to \cite[Lemma 3.7]{formuledivariazioneriem}, the curvature term in \eqref{jacobitipeequation384969} admits the explicit expression
\begin{equation}\label{riemcurvaturetermspresoda3949596}
        \left|h^\eps\right|^2+\ric^\eps(\n^\eps,\n^\eps)=\frac{|\tilde h^\hhh|^2+4 J(\vh)\alpha+(2n+2)\alpha^2}{1+\eps^2\alpha^2}+\frac{2\eps^2\left|\nabla^{\hhh,S}\alpha\right|^2}{\left(1+\eps^2\alpha^2\right)^2}+\frac{\eps^4\left(\s\alpha\right)^2}{\left(1+\eps^2\alpha^2\right)^3}\qquad \text{on }S\setminus S_0.
    \end{equation}
 
\subsection{General $t$-graphs}\label{sec_tgraph}
A (smooth) \emph{$t$-graph} in $\hh^n$ is a hypersurface of the form 
\begin{equation*}
    S=\{t=f(x,y)\,:\,(x,y)\in U\},
\end{equation*}
where $U\subseteq\rr^{2n}$ is an open set and $f\in C^\infty(\Om)$. It is well-known \cite{pmc1, MR5029815} that $p=(x,y,t)\in S_0$ if and only if 
\begin{equation*}
    Df(x,y)+F(x,y)=0,\qquad F(x,y)\coloneq (-y,x).
\end{equation*}
The (upward) $\eps$-Riemannian normal and horizontal normal of $S$ are
\begin{equation*}
    \n^\eps=-\sum_{i=1}^{2n}\left(\frac{D_if+F_i}{\sqrt{|Df+F|^2+\eps^2}}\right)Z_i+\left(\frac{\eps}{\sqrt{|Df+F|^2+\eps^2}}\right)\eps T,\qquad \vh=-\sum_{i=1}^{2n}\left(\frac{D_if+F_i}{|Df+F|}\right)Z_i,
\end{equation*}
where the latter is clearly defined only when $Df+F\neq 0$.
Moreover, the Riemannian and sub-Riemannian mean curvatures of $S$ and the fundamental function read respectively as 
\begin{equation*}
    H^\eps=-\divv\left(\frac{Df+F}{\sqrt{|Df+F|^2+\eps^2}}\right),\qquad   H^\hhh=-\divv\left(\frac{Df+F}{|Df+F|}\right),\qquad\alpha=\frac{1}{|Df+F|}.
\end{equation*}
If $\psi\in C^\infty(\Om)$, we identify it as a function on $S$ by $\psi(x,y,f(x,y))=\psi(x,y)$, and we extend it vertically by $\psi(x,y,t)=\psi(x,y)$. Accordingly, defining
\begin{equation*}
    g^{ij}_\eps=\delta_{ij}-\frac{(Df+F)_i(Df+F)_j}{|Df+F|^2+\eps^2},\qquad i,j=1,\ldots,2n,
\end{equation*}
then
\cite[(3.8)]{MR5029815} yields
\begin{equation}\label{normgradtgraficogenerale}
    \left|\nabla^{\eps,S}\psi\right|^2=|D\psi|^2-\frac{\left\langle D\psi,Df+F\right\rangle^2}{|Df+F|^2+\eps^2}.
\end{equation}
    Moreover, when $H^\eps\equiv 0$, \cite[(3.8)]{MR5029815} implies
    \begin{equation}\label{laplbelttgraficogenerale}
        \Delta^{\eps,S}\psi=\sum_{i,j=1}^{2n}g^{ij}_\eps D^2\psi_{ij}-\frac{2\left\langle D\psi,J(Df+F)\right\rangle}{|Df+F|^2+\eps^2}.
    \end{equation}
  \subsection{A special class of $t$-graphs} \label{sec_specialclass} Next, we specialize the above setting when $f\in C^\infty(\Om\times\rr^n)$ has the specific form
    \begin{equation*}
        f(x,y)=\langle x,y\rangle+u(x),\qquad x\in\Om,\,y\in\rr^n,
    \end{equation*}
    where $\Om\subseteq\rr^n$ is an open set. Notice that
    \begin{equation}\label{formaspecialecoso2939469}
        Df(x,y)=\left( y+Du(x),x\right),\qquad Df(x,y)+F(x,y)=\left(Du,2x\right).
    \end{equation}
    In particular, 
    $$ |Df+F|^2=|Du|^2+4|x|^2,$$
    whence 
    \begin{equation}\label{carsetforspecialgraf}
        S_0=\{(x,y,t)\in S\,:\,x=0,\,Du(x)=0\}.
    \end{equation}
   By the above formulas, and adopting the shortcuts
    \begin{equation}
    \label{eq:Weps}
      \W\coloneqq \sqrt{|Du|^2+4|x|^2},\qquad  \We\coloneqq \sqrt{|Du|^2+4|x|^2+\eps^2},
    \end{equation}
    then
    \begin{equation}\label{normalipergraficiridotti38586}
    \n^\eps=-\sum_{i=1}^{n}\left(\frac{u_i}{\We}\right)X_i-\sum_{i=1}^{n}\left(\frac{2x_i}{\We}\right)Y_i+\left(\frac{\eps}{\We}\right)\eps T,\qquad
    \vh=-\sum_{i=1}^{n}\left(\frac{u_i}{\W}\right)X_i-\sum_{i=1}^{n}\left(\frac{2x_i}{\W}\right)Y_i,
\end{equation}
and moreover
\begin{equation}\label{espressionecurvaturamediariempergarficispeciali}
     H^\eps=-\divv\left(\frac{Du}{\sqrt{|Du|^2+4|x|^2+\eps^2}}\right),\qquad   H^\hhh=-\divv\left(\frac{Du}{\sqrt{|Du|^2+4|x|^2}}\right),\qquad \alpha=\frac{1}{\W}.
\end{equation}
Finally, if $\psi\in C^\infty(\Om)$ is extended by $\psi(x,y)=\psi(x)$, and if $\psi$ is regarded as a function over $S$, \eqref{normgradtgraficogenerale} yields
\begin{equation}\label{gradientesulgraficopercosenony38486}
    \left|\nabla^{\eps,S}\psi\right|^2=|D\psi|^2-\frac{\left\langle D\psi,Du\right\rangle^2}{\We^2}.
\end{equation}
When in addition $H^\eps=0$, \eqref{laplbelttgraficogenerale} grants
\begin{equation}\label{equazioneperlaplacianograficoridotto38486854}
    \Delta^{\eps,S}\psi=\sum_{i,j=1}^{n}g^{ij}_\eps D^2\psi_{ij}+\frac{4\left\langle D\psi,x\right\rangle}{\We^2}=\Delta\psi-\frac{\left\langle D^2 \psi\,Du,Du\right\rangle}{\We^2}+\frac{4\left\langle D\psi,x\right\rangle}{\We^2}.
\end{equation}
\subsection{$X_1$-intrinsic graphs}
Another relevant class of hypersurfaces is that of intrinsic graphs \cite{MR1871966}. To this aim, set $\mathbb{W}:=\{x_1=0\}\subseteq\hh^n$. A (smooth) intrinsic
graph along $X_1$ is a hypersurface of the form
\[
    S=\{w\cdot\exp(\varphi(w)X_1):w\in U\},
\]
where $U\subseteq\mathbb{W}$ is an open set and
$\varphi\in C^\infty(U)$. We say that $S$ is \emph{entire} if
$U=\mathbb{W}$.
The following results provides a simple criterion to prove that hypersurface is an entire intrinsic graph.
\begin{proposition}\label{propcriteriointrinsic}
Let $G\in C^\infty(\hh^n)$, and set 
$S=\{G=0\}$. If $X_1G\geqslant 1$ on $\hh^n$, then $S$ is an entire, smooth intrinsic graph along $X_1$.
\end{proposition}
\begin{proof}
The map $(w,s)\mapsto w\cdot\exp(sX_1)$ is a smooth diffeomorphism
from $\mathbb{W}\times\mathbb{R}$ onto $\hh^n$. By assumption, recalling \cite[Lemma 2.12.1]{MR746308} and for every
$w\in\mathbb{W}$,
\[
    \frac{d}{ds}G\bigl(w\cdot\exp(sX_1)\bigr)
    =X_1G\bigl(w\cdot\exp(sX_1)\bigr)\geqslant 1.
\]
Therefore, $s\mapsto G(w\cdot\exp(sX_1))$ is strictly increasing
and tends to $\pm\infty$ as $s\to\pm\infty$. Hence it vanishes at
a unique $\varphi(w)\in\mathbb{R}$. To prove smoothness, fix $w_0\in\mathbb{W}$. The map
$(w,s)\mapsto G(w\cdot\exp(sX_1))$ is smooth, vanishes at
$(w_0,\varphi(w_0))$, and has nonzero derivative with respect to
$s$ at that point. Therefore, the implicit function theorem
provides an open neighborhood $U_0\subseteq\mathbb{W}$ of $w_0$
and a function $\varphi_0\in C^\infty(U_0)$ such that
$\varphi_0(w_0)=\varphi(w_0)$ and
\[
    G\bigl(w\cdot\exp(\varphi_0(w)X_1)\bigr)=0
    \qquad\text{for every }w\in U_0.
\]
For every $w\in U_0$, both $\varphi_0(w)$ and $\varphi(w)$
are zeros of $s\mapsto G(w\cdot\exp(sX_1))$. By the uniqueness
proved above, $\varphi_0=\varphi$ on $U_0$. Therefore
$\varphi\in C^\infty(\mathbb{W})$, and the thesis follows.
\end{proof}
\section{Proof of the main results}
Throughout this section, we fix $n\geq 2$.
\subsection{The regularized equation}\label{sec_reg_eq}
 For $0\leq\eps\leq 1$, set
 \begin{equation}
 \label{eq:Qeps}
    \Qe u\coloneqq\divv\left(\frac{Du}{\sqrt{|Du|^2+4|x|^2+\eps^2}}\right)= \frac{(|Du|^2+4|x|^2+\eps^2)\Delta u-\left\langle D^2 u\, Du,Du\right\rangle-4\langle Du,x\rangle}{\left(|Du|^2+4|x|^2+\eps^2\right)^\frac{3}{2}},
\end{equation}
\begin{equation}
\label{eq:Neps}
      \Ne u \coloneqq \left(|Du|^2+4|x|^2+\eps^2\right)^\frac{3}{2}\Qe u=(|Du|^2+4|x|^2+\eps^2)\Delta u-\left\langle D^2 u\, Du,Du\right\rangle-4\langle Du,x\rangle.  
\end{equation}

When $\eps\ge 0$, both operators are classically well-defined. Moreover, $\Qe u=0$ if and only if $\Ne u=0$. In the following, we always assume $\eps\in(0,1]$. For $x,p\in\rr^n$, set 
\begin{equation}\label{definizione_matrici_eqregeps}
    \Ae(x,p)=(|p|^2+4|x|^2+\eps^2)I_n-p\otimes p,\qquad b(x,p)=-4\langle p,x\rangle.
\end{equation}
Then 
\begin{equation*}
    \Ne u=\sum_{i,j=1}^n\Ae_{ij}(x,Du)D^2 u_{ij}+b(x,Du).
    \end{equation*}
    Notice that 
    \begin{equation}\label{derivatematriceae384869}
        \partial _{x_i}\Ae_{jk}=8x_i\delta_{jk},\qquad  \partial _{p_i}\Ae_{jk}=2p_i\delta_{jk}-p_j\delta_{ik}-p_k\delta_{ij},\qquad i,j,k=1,\ldots,n.
    \end{equation}
    For future purposes, for $\sigma\in[0,1]$, define also 
  \begin{equation*}
    \Nes u=\sum_{i,j=1}^n\Ae_{ij}(x,Du)D^2 u_{ij}+\sigma b(x,Du),
    \end{equation*}
    and define $\Qes$ accordingly. Clearly $\mathcal Q^{\eps,1}=\Qe$ and $\mathcal N^{\eps,1}=\Ne$.
Following the notation of \cite{MR1814364}, denote by $\lambda^\eps(x,p)$ and $\Lambda^\eps(x,p)$ the lowest and the highest eigenvalues of $\Ae(x,p)$. It is simple to check that 
\begin{equation}\label{eigenvaluesofaeps}
    \lambda^\eps(x,p)=4|x|^2+\eps^2,\qquad \Lambda^\eps(x,p)=|p|^2+4|x|^2+\eps^2,\qquad\frac{\Lambda^\eps(x,p)}{\lambda^\eps(x,p)}=\frac{|p|^2+4|x|^2+\eps^2}{4|x|^2+\eps^2}.
\end{equation}
Therefore, according to the definitions in \cite[Chapter 10]{MR1814364}, $\Qes$ and $\Nes$ are second-order, quasi-linear elliptic operators. Moreover, for any bounded domain $\Om\subseteq\rr^n$ and for any $u\in C^2(\Om)\cap C(\overline\Om)$ such that $Du$ is bounded on $\Om$, then $\Qes$ and $\Nes$ are uniformly elliptic with respect to $u$ over $\Om$. Accordingly, the following comparison principle holds. 
\begin{proposition}\label{prop_comparison_principle}
   Let $\Om\subseteq\rr^n$ be a bounded domain. Let $u,v\in C^2(\Om)\cap C(\overline\Om)$. Assume $\Qes u\geq0$, $\Qes v\leq 0$ on $\Om$. If $u\leq v$ on $\partial\Om$, then $u\leq v$ on $\Om$. In particular, if $\Qes u,\Qes v=0$ on $\Om$ and $u=v$ on $\partial\Om$, then $u= v$ on $\Om$. The same conclusions hold when $\Qes$ is replaced by $\Nes$.
\end{proposition}
\begin{proof}
    Since $u\in C^1(\Om)$, $\Qes$ and $\Nes$ are locally uniformly elliptic with respect to $u$, in the sense of \cite[Chapter 10]{MR1814364}. The thesis then follows by \cite[Theorem 10.1]{MR1814364}.
\end{proof}
\subsection{The regularized Dirichlet problems on balls}
In this section, we solve the Dirichlet problem associated with $\Qe$, when the domain is a ball and the boundary datum has a very specific shape. Throughout the paper, we denote by $B_R=B_R(0)$ the Euclidean ball centered at $0$.
\begin{proposition}\label{propositiondiriboundcondex}
Let $R,K>0$ and $\eps\in(0,1]$. The problem
\begin{equation}\label{eq_diri_ball_eps}
 \Qe u=0\quad\text{on }B_R,
 \qquad u=Kx_1\quad\text{on }\partial B_R
\end{equation}
has a unique solution $u\in C^\infty(\overline B_R)$. Moreover, $u$ is odd in $x_1$, and invariant under orthogonal transformations of $x_2,\ldots,x_n$. 
\end{proposition}
\begin{proof}
    First, we show existence. We rely on the well-known Leray-Schauder fixed point approach. To this aim, 
    Fix any $\alpha\in(0,1)$. By \cite[Theorem 13.8]{MR1814364}, \eqref{eq_diri_ball_eps} is solvable in $C^{2,\alpha}(\overline B_R)$ provided that solutions to 
    \begin{equation}\label{eq_diri_ball_eps_sigma}
 \Nes u=0\quad\text{on }B_R,
 \qquad u=\sigma Kx_1\quad\text{on }\partial B_R
\end{equation}
are bounded in $C^1(\overline B_R)$ uniformly in $u$ and $\sigma\in[0,1]$. Fix any such solution. When $\sigma=0$, then $u\equiv 0$ solves \eqref{eq_diri_ball_eps_sigma}, and it is the unique solution by \Cref{prop_comparison_principle}. Therefore, we assume that $\sigma\ge 0$. 

\smallskip
\noindent \textbf{Step 1.} We first show that $u\in C^\infty(\overline B_R)$. Indeed, setting $\mathcal A^\eps(x)=\Ae(x,Du(x))$, then $u$ satisfies 
\begin{equation}\label{linear_elliptic_forreg}
\sum_{i,j=1}^n\mathcal A^\eps(x)_{ij}D^2 u_{ij}-4\sigma \langle Du,x\rangle=0\qquad \text{on $B_R$.}
\end{equation}
The latter, is a second-order \emph{linear} equation without zero-order terms. Since $u\in C^{2,\alpha}(\overline B_R)$, then \eqref{eigenvaluesofaeps} implies that \eqref{linear_elliptic_forreg} is uniformly elliptic, and that $\mathcal A^\eps_{ij}\in C^{1,\alpha}(\overline B_R)$. Therefore, \cite[Theorem 6.19]{MR1814364} grants $u\in C^{3,\alpha}(\overline B_R)$. Bootstrapping the above argument, $u\in C^\infty(\overline B_R).$

\smallskip
\noindent \textbf{Step 2.} Next, we derive uniform estimates in $C^1(\overline B_R)$. First, notice that the constant functions $-KR$ and $KR$ are solutions to the equation. Moreover, $-KR\leq u\leq KR$ on $\partial B_R.$ Then, \Cref{prop_comparison_principle} yields $-KR\leq u\leq KR$ on $\overline B_R$, i.e. $|u|\leq KR$ on $\overline B_R$. 

\smallskip
\noindent \textbf{Step 3.} In order to achieve boundary gradient estimates, we exhibit an upper and a lower barrier realizing the boundary datum. Accordingly, set $b_\pm\coloneqq \sigma Kx_1\pm(R^2-|x|^2)$. Clearly, $b_+=b_-=u$ on $\partial B_R$. We show that $\Nes b_+\leq 0$. Indeed, since
\begin{equation*}
    Db_+=(\sigma K-2x_1,-2x_2,\ldots,-2x_n),\quad Db_-=(\sigma K+2x_1,2x_2,\ldots,2x_n),\quad D^2 b_+=-2I_n,\quad D^2 b_-= 2I_n,
\end{equation*}
then 
\begin{equation*}
    \begin{split}
        \Nes b_+\overset{\eqref{definizione_matrici_eqregeps}}&{=}-2n\big(|Db_+|^2+4|x|^2+\eps^2\big)+2|Db_+|^2-4\sigma\langle D b_+,x\rangle \\
        &=-2(n-1)|Db_+|^2-4\sigma\langle D b_+,x\rangle-8n|x|^2-2n\eps^2\\
        &=-2(n-1)\left|Db_++\frac{\sigma x}{n-1}\right|^2+\left(\frac{2\sigma^2}{n-1}-8n\right)|x|^2-2n\eps^2\\
        &\leq 0.
    \end{split}
\end{equation*}
Similarly,
\begin{equation*}
    \begin{split}
         \Nes b_-\overset{\eqref{definizione_matrici_eqregeps}}&{=}2n\big(|Db_-|^2+4|x|^2+\eps^2\big)-2|Db_-|^2-4\sigma\langle D b_-,x\rangle \\
        &=2(n-1)|Db_-|^2-4\sigma\langle D b_-,x\rangle+8n|x|^2+2n\eps^2\\
        &=2(n-1)\left|Db_--\frac{\sigma x}{n-1}\right|^2+\left(8n-\frac{2\sigma^2}{n-1}\right)|x|^2+2n\eps^2\\
        &\geq 0.
    \end{split}
\end{equation*}
Therefore, \Cref{prop_comparison_principle} yields $b_- \leq u\leq b_+$ on $\overline B_R$. Since both barriers coincide with $u$ on $\partial B_R$ and $|D b_\pm|\leq K+2R$ on $\overline B_R$, we conclude that $|D u|\leq K+2R$ on $\partial B_R$. 

\smallskip
\noindent \textbf{Step 4.} 
To conclude, we show that 
\begin{equation}\label{quasi_max_principle}
    \max_{\overline B_R} |Du|^2\leq \max\left\{ \max_{\partial B_R} |Du|^2,4R^2\right\}.
\end{equation}
By contradiction, assume \eqref{quasi_max_principle} does not hold. Let $\tilde x\in B_R$ be a point where \eqref{quasi_max_principle} fails. Then, we may assume that $\tilde x$ is an interior maximum point of $w$. For notational simplicity, set $w=|Du|^2$. Therefore
\begin{equation}\label{maximum_point_of_w_in_max_principle}
    Dw=2 D^2 u\,Du=2\left(\sum_{i=1}^n u_{1i} u_i,\ldots,\sum_{i=1}^n u_{ni} u_i\right)=0,\qquad D^2 w\leq 0\qquad\text{at $\tilde x$.}
\end{equation}
In particular, 
\begin{equation}\label{equazione_in_pt_massimo_w_max_principle_34578}
(w+4|x|^2+\eps^2)\Delta u\overset{\eqref{eq_diri_ball_eps_sigma}}{=}4\sigma \langle Du,x\rangle\qquad\text{at $\tilde x$.}
\end{equation}
By direct computation, at $\tilde x$,
\begin{equation}\label{eq_ausiliarie_28486948375}
    \partial _i\langle Du,x\rangle =u_i+\sum_{j=1}^nu_{ij} x_j,\qquad \partial _i \Ae_{jk}\coloneqq \partial_i\left(\Ae_{jk}(x,Du)\right)\overset{\eqref{maximum_point_of_w_in_max_principle}}{=}8x_i\delta_{jk}-u_{ij}u_k-u_{ik}u_j\qquad i=1,\ldots, n.
\end{equation}

Therefore, at $\tilde x$,
\begin{equation*}
    \begin{split}
        0\overset{\eqref{eq_diri_ball_eps_sigma}}&{=}\sum_{i=1}^n\partial_i\left(\Nes u\right)\,u_i\\
        \overset{\eqref{maximum_point_of_w_in_max_principle},\eqref{eq_ausiliarie_28486948375}}&{=}\sum_{i,j,k=1}^n\partial_i\Ae_{jk}\,u_{jk}\,u_i+\sum_{i,j,k=1}^n\Ae_{jk}\,u_{ijk}\,u_i-4\sigma w\\
        \overset{\eqref{maximum_point_of_w_in_max_principle},\eqref{eq_ausiliarie_28486948375}}&{=}8\Delta u\langle Du,x\rangle+\sum_{i,j,k=1}^n\Ae_{jk}\,u_{ijk}\,u_i-4\sigma w,
    \end{split}
\end{equation*}
whence
\begin{equation}\label{eq_ausiliariaperderivataaaaaa384954}
    \sum_{i,j,k=1}^n\Ae_{jk}\,u_{ijk}\,u_i=4\sigma w-8\Delta u\langle Du,x\rangle\qquad\text{at $\tilde x$.}
\end{equation}
Noticing that $\trace\left(\Ae\cdot D^2 w\right)\leq 0$ at $\tilde x$ by \eqref{maximum_point_of_w_in_max_principle}, then, at $\tilde x$,  
\begin{equation*}
\begin{split}
      0&\geq \frac{1}{2}\sum_{i,j=1}^n\Ae_{ij}w_{ij}\\
      &=\sum_{i,j,k=1}^n\Ae_{ij}\partial_i\big(u_{jk} u_k\big)\\
      &=\sum_{i,j,k=1}^n\Ae_{ij}u_{ijk} u_k+\sum_{i,j,k=1}^n\Ae_{ij}u_{jk} u_{ik}\\
      \overset{\eqref{eq_ausiliariaperderivataaaaaa384954}}&{=}4\sigma w-8\Delta u\langle Du,x\rangle+\sum_{i,j,k=1}^n\Ae_{ij}u_{jk} u_{ik}\\
      \overset{\eqref{maximum_point_of_w_in_max_principle}}&{=}4\sigma w-8\Delta u\langle Du,x\rangle+(w+4|x|^2+\eps)^2|D^2 u|^2.
\end{split}   
\end{equation*}
In particular, $4\sigma w-8\Delta u\langle Du,x\rangle\leq 0$ at $\tilde x$. On the other hand, at $\tilde x$,
\begin{equation*}
    \begin{split}
        4\sigma w-8\Delta u\langle Du,x\rangle\overset{\eqref{equazione_in_pt_massimo_w_max_principle_34578}}&{=}4\sigma \left(w-\frac{8\langle Du,x\rangle^2}{w+4|x|^2+\eps^2}\right)\\
        &\geq 4\sigma w\left(1-\frac{8|x|^2}{w+4|x|^2+\eps^2}\right)\\
        &=4\sigma w\left(\frac{w-4|x|^2+\eps^2}{w+4|x|^2+\eps^2}\right),
    \end{split}
\end{equation*}
Since $\sigma\ge 0$ and $w(\tilde x)\ge 4R^2$ by the choice of $\tilde x$, then   $4\sigma w-8\Delta u\langle Du,x\rangle\ge 0$ at $\tilde x$, a contradiction. Therefore \eqref{quasi_max_principle} holds. Combining the previous step, we conclude that \eqref{eq_diri_ball_eps} is solvable in $C^{2,\alpha}(\overline B_R)$.

\smallskip
\noindent \textbf{Step 5.} Let $u\in C^{2,\alpha}(\overline B_R)$ solve \eqref{eq_diri_ball_eps}. By Step 1, $u\in C^{\infty}(\overline B_R)$. Moreover, $u$ is unique by \Cref{prop_comparison_principle}. Since $x\mapsto -u(-x_1,x_2,\ldots,x_n)$ solves  \eqref{eq_diri_ball_eps} as well, then $u$ is odd in $x_1$ by uniqueness. Similarly, it is simple to check that $x\mapsto u(x_1, R(x_2,\ldots,x_n))$ solves \eqref{eq_diri_ball_eps} for any orthogonal transformation $R$ of $x_2,\ldots,x_n$. Again, the thesis follows by uniqueness. 
\end{proof}
We point out that Step 4 in the previous proof holds regardless of the given boundary condition in \eqref{eq_diri_ball_eps}. Accordingly, its very same proof provides the following estimate, which will be used in the next sections. 
\begin{proposition}\label{quasimaximumprincipleproposition}
    Let $0\le \varrho\leq R$, $\eps\in (0,1]$. Let $u\in C^\infty(\overline B_R)$ satisfy $\Qe u=0$ on $B_R$. Then  
\begin{equation}\label{quasi_max_principle_in_proposition}
    \max_{\overline B_\varrho} |Du|^2\leq \max\left\{ \max_{\partial B_\varrho} |Du|^2,4\varrho^2\right\}.
\end{equation}
\end{proposition}

\subsection{Uniform height estimates}\label{erfnjtrvijng}
We provide height estimates for the solutions to \eqref{eq_diri_ball_eps},
uniformly with respect to $R>0$ and $\eps\in(0,1]$, provided that a precise boundary condition is prescribed. Accordingly, set
\begin{equation}\label{defigammaecn}
     \gamma\coloneqq\frac{\sqrt{(n-1)^2+4}-(n-1)}4,
 \qquad
 C_n\coloneqq\begin{cases}3,&n=2,\\ \frac{2^\gamma}{2\gamma},&n\geq3,\end{cases}
\end{equation}
and choose the boundary slope
\begin{equation}\label{definizionekappaperheightestimates3948568}
 K(R)\coloneqq C_n(1+R^2)^\gamma.
\end{equation}
\begin{proposition}\label{propunistimamain3948573883derivata}
For every $R>0$ and $\eps\in(0,1]$,
the solution $u\in C^\infty(\overline B_R)$ of \eqref{eq_diri_ball_eps}
with $K=K(R)$ as in \eqref{definizionekappaperheightestimates3948568}
satisfies
\begin{equation}\label{uniformestimatesproprincheigh}
 x_1\leq u(x)\leq 2(4^\gamma C_n)^{\frac{1}{1-2\gamma}} x_1(1+|x|^2)^\gamma
 \qquad\text{on }\overline B_R\cap\{x_1\geq0\}.
\end{equation}
\end{proposition}
The proof of \Cref{propunistimamain3948573883derivata} will follow by exhibiting an upper and a lower barrier for $\Ne$ which are uniform with respect to $\eps\in(0,1]$ and $R\ge0$, and which are compatible with the boundary condition dictated by \eqref{definizionekappaperheightestimates3948568}. In the following computations, we use
\begin{equation}\label{eq:barrier-gamma}
 4\gamma^2+2(n-1)\gamma=1,\qquad
 0<\gamma<\tfrac13,\qquad 0<\gamma<\tfrac14\text{ if }n\geq 3,\qquad
 2\big(1-(n+1)\gamma\big)=(1-2\gamma)^2.
\end{equation}
\begin{proposition}\label{prop:barrier-profile}
Let $h:\rr_+\to\rr$ be a smooth function. Set $v(x)=x_1h(|x|^2)$. Then 
\begin{equation}\label{eq:barrier-profile}
 \frac{\Ne v}{x_1}=C^\eps(|x|^2)+x_1^2D(|x|^2),\qquad
 \frac{\Delta v}{x_1}=2(n+2)h'+4|x|^2 h''\qquad\text{on } \{x_1\ge 0\},
\end{equation}
where
\begin{align*}
 C^\eps(s)&\coloneqq 4s(h^2+4s+\eps^2)h''
       +[2(n-1)h^2+8(n+1)s+2(n+2)\eps^2]h'
       -8s h(h')^2-4h,\\
 D(s)&\coloneqq 8nh(h')^2+8(n-1)s(h')^3-4h^2h''.
\end{align*}
In particular, if $h>0$, $h'>0$ and $h''<0$, then $D>0$ and 
\begin{equation}
\label{eq:cepsin}
C^\eps<\frac{\Ne v}{x_1}  \leq C^\eps+sD
\end{equation}
where
\begin{equation}\label{eq:barrier-endpoint}
 \begin{split}
 C^\eps+sD={}&4s(4s+\eps^2)h''
       +[8(n+1)s+2(n+2)\eps^2]h'-4h+2(n-1)(h+2s h')^2h'.
 \end{split}
\end{equation}
\end{proposition}
\begin{proof}
Direct differentiation gives
\begin{align*}
 Dv&=he_1+2x_1h'x,\\
 D^2v&=2h'(e_1\otimes x+x\otimes e_1+x_1I_n)+4x_1h''x\otimes x,\\
 |Dv|^2&=h^2+4x_1^2hh'+4x_1^2|x|^2(h')^2,\qquad
 \langle x,Dv\rangle=x_1(h+2|x|^2 h').
\end{align*}
Tracing the second equation gives the Laplacian in \eqref{eq:barrier-profile}. Moreover,
\[
 \langle D^2v\,Dv,Dv\rangle
 =4h'v_1\langle x,Dv\rangle+2x_1h'|Dv|^2
   +4x_1h''\langle x,Dv\rangle^2.
\]
Therefore, the rest of the claims easily follow.
\end{proof}

\begin{proposition}\label{prop:barrier-upper}
For every $n\geq2$ and for every $\eps\in(0,1]$, the smooth function
\[
 U(x)\coloneqq x_1\left[1+\left(1+\frac{|x|^2}{4}\right)^\gamma\right]
\]
satisfies
\begin{equation}\label{eq:barrier-upper}
 \Ne U\leq 0,\qquad
 4^{-\gamma}x_1(1+|x|^2)^\gamma
 \leq U(x)\leq2x_1(1+|x|^2)^\gamma\qquad \text{on $\{x_1>0\}$.}
\end{equation}
\end{proposition}
\begin{proof}
Set $h(s)=1+(1+s/4)^\gamma$ and $t=1+s/4\geq1$. In the following, for the sake of notational simplicity, we let $s=|x|^2.$ Thus
\begin{equation*}
    h'=\frac\gamma4t^{\gamma-1}>0,\qquad h''=\frac{\gamma(\gamma-1)}{16}t^{\gamma-2}<0.
\end{equation*}
Since $x_1^2\leq s$, \eqref{eq:barrier-profile}, \eqref{eq:cepsin} and \eqref{eq:barrier-endpoint}
give
\[
 \frac{\Ne U}{x_1}\leq
 16s^2h''+8(n+1)s h'-4h
 +\eps^2[4s h''+2(n+2)h']
 +2(n-1)(h+2s h')^2h'.
\]
By \eqref{eq:barrier-gamma}, the first three terms combine as
\[
 16s^2h''+8(n+1)s h'-4h
 =-4-4t^{\gamma-2}[1+(1-2\gamma)^2(t-1)]<-4.
\]
Indeed, expansion in $t-1$ cancels its quadratic coefficient, while the
linear coefficient is $8(n+1)\gamma-8=-4(1-2\gamma)^2$.
Next, $(n-1)\gamma<1/2$ and $\gamma<1/3$ imply
\[
 \eps^2[4s h''+2(n+2)h']
 \leq2(n+2)h'\leq\frac{(n+2)\gamma}{2}<\frac34.
\]
Finally, $h+2s h'\leq(2+2\gamma)t^\gamma$, whence
\[
 2(n-1)(h+2s h')^2h'
 \leq\frac{(n-1)\gamma}{2}(2+2\gamma)^2t^{3\gamma-1}<2.
\]
Here $(2+2\gamma)^2<(8/3)^2<8$ and $3\gamma-1<0$.
Consequently $\Ne U/x_1<-4+3/4+2<0$. The growth bounds follow from
$(1+s)/4\leq1+s/4\leq1+s$ and $1\leq(1+s)^\gamma$.
\end{proof}

\begin{proposition}\label{prop:barrier-lower-higher}
For every $n\geq3$ and for every $\eps\in(0,1]$, the smooth function
\[
 u_-(x)\coloneqq x_1h(|x|^2),\qquad
 h(s)\coloneqq1+\frac{(1+2s)^\gamma-1}{2\gamma}
\]
satisfies
\begin{equation}\label{eq:barrier-lower-higher}
 \Ne u_-\geq 0,\qquad
 x_1\leq u_-(x)\leq\frac{2^\gamma}{2\gamma}x_1(1+|x|^2)^\gamma\qquad\text{on }\{x_1\ge 0\}.
\end{equation}
\end{proposition}
\begin{proof}
We have $h(0)=1$, $h'=(1+2s)^{\gamma-1}>0$ and
$h''=2(\gamma-1)(1+2s)^{\gamma-2}<0$.
Thus $D>0$ in \eqref{eq:barrier-profile}. Therefore we get \eqref{eq:cepsin} for $v=u_-$. Since 
\[
\Ne u_-= \mathcal{N}^0u_-+ \eps^2 \Delta  u_-,
\]
and \eqref{eq:barrier-profile} yields
\[
 \frac{\Delta u_-}{x_1}
 =2(1+2s)^{\gamma-2}[n+2+(2n+4\gamma)s]>0,
\]
we have $\frac{\Ne v}{x_1} \ge C^0$.
If $C^0\geq0$, this proves the strict subsolution inequality.
For this calculation set $t=1+2s$ and $c=1-2\gamma$, so that
$h=(t^\gamma-c)/(2\gamma)$. Direct substitution gives
\begin{equation}
\label{eq:2C0start}
 \begin{split}
 2\gamma C^0={}&4c-4t^{\gamma-2}[1+c^2(t-1)]\\
 &+\frac1\gamma t^{\gamma-2}(t^\gamma-c)
 \big\{(n-1)(t^\gamma-c)t
       -2[(\gamma+1)t^\gamma+(\gamma-1)c](t-1)\big\}.
 \end{split}
\end{equation}
The first line collects $16s^2h''+8(n+1)s h'-4h$; the second collects
$4s h^2h''+2(n-1)h^2h'-8s h(h')^2$, each multiplied by $2\gamma$. The defining relation $4\gamma^2+2(n-1)\gamma=1$ gives
\[
  \frac1\gamma=4+4\gamma+2(n-3),\qquad
  2(n-3)\gamma=1-4\gamma-4\gamma^{2}.
\]

\medskip
\noindent
We first simplify the second line of \eqref{eq:2C0start}. Let
$R:=\frac1\gamma\,t^{\gamma-2}(t^\gamma-c)\,B$, where
\[
  B=(n-1)(t^\gamma-c)\,t-2\big((\gamma+1)t^\gamma+(\gamma-1)c\big)(t-1).
\]
Then $B=2\Theta+(n-3)\,t\,(t^\gamma-c)$ with
\[
  \Theta:=t(t^\gamma-c)-\big((\gamma+1)t^\gamma+(\gamma-1)c\big)(t-1)
  =-\gamma t^{\gamma+1}+(\gamma+1)t^{\gamma}-\gamma c\,t+(\gamma-1)c .
\]
Multiplying by $t^\gamma-c$ and collecting terms gives the identity
\begin{equation}\tag{L}
  (t^\gamma-c)\,\Theta=(t^\gamma-c)^2-\gamma\,(t-1)\,(t^{2\gamma}-c^{2}).
\end{equation}
Indeed, the left-hand side equals
$-\gamma t^{2\gamma+1}+(\gamma+1)t^{2\gamma}-2c\,t^{\gamma}+\gamma c^2 t+(1-\gamma)c^2$,
and the right-hand side expands to the same expression. Hence
\[
  R=\frac{2+(n-3)t}{\gamma}\,t^{\gamma-2}(t^\gamma-c)^2-2(t-1)\,t^{\gamma-2}(t^{2\gamma}-c^2).
\]

\medskip
\noindent
The first line of \eqref{eq:2C0start} is $4c-4t^{\gamma-2}\big[1+c^2(t-1)\big]$. In the sum with $R$ the terms in $c^2$ combine as $-4c^2+2c^2=-2c^2$, so
\begin{equation}\label{A}
  2\gamma C^0=4c-4t^{\gamma-2}-2c^2(t-1)\,t^{\gamma-2}-2(t-1)\,t^{3\gamma-2}
  +\frac{2+(n-3)t}{\gamma}\,t^{\gamma-2}(t^\gamma-c)^2 .
\end{equation}

\medskip
\noindent
Next we isolate the dependence on $n-3$. Since $\frac2\gamma=8+8\gamma+4(n-3)$,
\[
  \frac{2+(n-3)t}{\gamma}\,t^{\gamma-2}(t^\gamma-c)^2
  =(8+8\gamma)\,t^{\gamma-2}(t^\gamma-c)^2
  +(n-3)\Big[4\,t^{\gamma-2}(t^\gamma-c)^2+\tfrac1\gamma\, t^{\gamma-1}(t^\gamma-c)^2\Big].
\]
Only the piece $4(n-3)\,t^{\gamma-2}(t^\gamma-c)^2$ is rewritten, so that its $(n-3)$-part has a manifestly nonnegative form. Using $t^\gamma-c=t^\gamma-1+2\gamma$, we have
\[
  (t^\gamma-c)^2-(t^\gamma-1)^2-\gamma(t-1)=\gamma\,[\,4t^\gamma-t-3+4\gamma\,],
\]
and therefore
\[
  4(n-3)\,t^{\gamma-2}(t^\gamma-c)^2
  =4(n-3)\,t^{\gamma-2}\big[(t^\gamma-1)^2+\gamma(t-1)\big]
  +4(n-3)\gamma\big[\,4t^{2\gamma-2}-t^{\gamma-1}-(3-4\gamma)\,t^{\gamma-2}\big].
\]

\medskip
\noindent
Inserting this into \eqref{A} and expanding $(t^\gamma-c)^2=t^{2\gamma}-2c\,t^\gamma+c^2$, the coefficients of the powers of $t$ in the part not containing $(n-3)$ are:
\begin{align*}
   -2 & \quad \text{for }\quad t^{3\gamma-1},\\
   2+(8+8\gamma)=10+8\gamma &\quad \text{for }\quad t^{3\gamma-2},\\
  -16(1+\gamma)c+16(n-3)\gamma=-8(1+2\gamma) &\quad \text{for }\quad  t^{2\gamma-2},\\
  -2c^2-4(n-3)\gamma=-4(1-4\gamma) &\quad \text{for }\quad t^{\gamma-1}, \\
  -4+(10+8\gamma)c^2-4(n-3)\gamma(3-4\gamma)=0 &\quad \text{for }\quad t^{\gamma-2},\\
  4c=4(1-2\gamma)&\quad \text{for }\quad t^{0}.
\end{align*}
The vanishing of the coefficient of $t^{\gamma-2}$ follows from
$(10+8\gamma)c^2=10-32\gamma+8\gamma^2+32\gamma^3=4(n-3)\gamma(3-4\gamma)+4$.
We thus obtain 
\begin{equation}\label{giustificazione}
 \begin{split}
 2\gamma C^0={}&4(1-2\gamma)-2t^{3\gamma-1}
 +(10+8\gamma)t^{3\gamma-2}-8(1+2\gamma)t^{2\gamma-2}
 -4(1-4\gamma)t^{\gamma-1}\\
 &+(n-3)\left\{
 4t^{\gamma-2}[(t^\gamma-1)^2+\gamma(t-1)]
 +\frac1\gamma t^{\gamma-1}(t^\gamma-1+2\gamma)^2\right\}.
 \end{split}
\end{equation}
The expression in braces is $\geq 0$ because $t\ge1$ and $\gamma>0$. then we have 
\[
 2\gamma C^0 \geq 4(1-2\gamma)-2t^{3\gamma-1}
 +(10+8\gamma)t^{3\gamma-2}-8(1+2\gamma)t^{2\gamma-2}
 -4(1-4\gamma)t^{\gamma-1}
\]
Setting $q=t^{\gamma-1}$, so that $q^2=t^{2\gamma-2}$ we get 
\[
  2\gamma C^0 \geq 2\big(1-t^{3\gamma-1}\big)+2(1-4\gamma)(1-q)^2+(10+8\gamma)\,t^{2\gamma-2}\big(t^{\gamma}-1\big),
\]
which is checked by comparing coefficients: the constant term is $2+2(1-4\gamma)=4(1-2\gamma)$; the coefficient of $q$ is $-4(1-4\gamma)$; the coefficient of $q^2$ is $2(1-4\gamma)-(10+8\gamma)=-8(1+2\gamma)$. All three summands are nonnegative, since $3\gamma<1$, $\gamma<\tfrac14$ (this is another place where $n\geq 3$ is used), $\gamma>0$ and $t\geq1$. Hence
\[
  2\gamma C^0\ \geq \ 0 .
\]
The other assertions are immediate from the formula.
\end{proof}
In the previous proof, the constraint $n\geq 3$ was essential to discard the last non-negative contribution in \eqref{giustificazione}. Accordingly, the case $n=2$ must be treated separately.
\begin{proposition}\label{prop:barrier-lower-planar}
Let $n=2$ and $\delta=3\gamma-2$. For every $\eps\in(0,1]$, the smooth function
\[
 u_-(x)=x_1(1+8|x|^2)^\gamma+3x_1^3(1+8|x|^2)^\delta
\]
satisfies
\begin{equation}\label{subsoluzionejernfvojnttijvnrtvjnr}
   \Ne u_-\geq 0,\qquad x_1\leq u_-\leq 3x_1(1+|x|^2)^\gamma,  \qquad \text{on }\{x_1\ge 0\}.
\end{equation}
\end{proposition}
\begin{proof}
 Since $n=2$, a direct computation yields
\begin{equation}\label{eq_planar_exponent_bounds}
  4\gamma^2+2\gamma=1,\qquad \frac{309}{1000}<\gamma<\frac{31}{100},\qquad
 -\frac{11}{10}<\delta<-1.
\end{equation}
When $x_1\geq 0$, $u_-\geq x_1$ by definition.
Since 
\begin{equation}\label{splittingdeglioperatioriermerijn}
    \Ne u_-= \mathcal{N}^0u_-+ \eps^2 \Delta  u_-,
\end{equation}
we estimate separately the signs of the two differential operators.

\smallskip
\noindent\textbf{Step 1.}
Fix $x\in\rr^2$ with $x_1>0$. 
Define $V(s)=(1+8s)^\gamma$ and $W(s)=(1+8s)^\delta$, so that 
\begin{equation*}
    u_-(x)=x_1 V(|x|^2)+3x_1^3 W(|x|^2).
\end{equation*} 
 Then,
\begin{equation}\label{variederivatrebdin39gfnvitvnei}
 \begin{split}
(u_-)_1&=V+2x_1^2V'+9x_1^2W+6x_1^4W',\\
 (u_-)_2&=2x_1x_2(V'+3x_1^2W'),\\
 (u_-)_{11}&=6x_1V'+4x_1^3V''+18x_1W+42x_1^3W'+12x_1^5W'',\\
 (u_-)_{22}&=2x_1V'+4x_1x_2^2V''+6x_1^3W'+12x_1^3x_2^2W'',\\
 (u_-)_{12}&=2x_2V'+4x_1^2x_2V''+18x_1^2x_2W'+12x_1^4x_2W''.
 \end{split}
\end{equation}
We claim that 
\begin{equation}\label{eq_planar_numerator_identity_manonevero}
\mathcal N^0u_-
 \ge 0.
\end{equation}
Notice that 
\[
 \mathcal N^0u_-=\big((u_-)_2^2+4|x|^2\big)(u_-)_{11}
 +\big((u_-)_1^2+4|x|^2\big)(u_-)_{22}
 -2(u_-)_1(u_-)_2(u_-)_{12}-4\big(x_1(u_-)_1+x_2(u_-)_2\big).
\]
Set
\[
 z=(1+8|x|^2)^{-1},\quad \eta=z^{1-2\gamma},\quad
 w=(1-z)\eta,\quad \tau=\frac{x_1^2}{|x|^2}.
\]
Notice that $z,\eta\ge 0$. Define the following polynomials with respect to $z:$
\begin{align*}
 a(z)&=(32\gamma-12)z^2+(27-80\gamma)z+64\gamma-19,\\
 b(z)&=\tfrac12\big[(234\gamma-99)z^3+(173-448\gamma)z^2
                   +(194\gamma-82)z+52\gamma\big],\\
 C_1(z)&=3\big[(2\gamma-15)z+16\gamma+11\big],\\
 D_1(z)&=9(z-1)\big[(20\gamma-7)z+2\gamma+6\big],\\
 C_2(z)&=-\tfrac{27}{8}\big[(14\gamma+7)z+1-32\gamma\big],\\
 D_2(z)&=\tfrac98(z-1)\big[(348\gamma-99)z+34-6\gamma\big],\\
 C_3(z)&=-\tfrac{81}{64}\big[(30\gamma-1)z+13-48\gamma\big],\\
 D_3(z)&=\tfrac{27}{64}(z-1)\big[(612\gamma-199)z+46-126\gamma\big].
\end{align*}
Define
\begin{equation*}
 P(z,w,\tau)=(1-\tau)a+\tau b
       +w(\tau C_1+\tau^2D_1)+w^2(\tau^2C_2+\tau^3D_2)
       +w^3(\tau^3C_3+\tau^4D_3).
\end{equation*}
Notice that $P$ is a polynomial function with respect to $(z,w,\tau)$.
A direct computation shows that 
\begin{equation}\label{duecoseederftrgrrgt}
    x_1^2=\frac{\tau(1-z)}{8z},\qquad x_2^2=\frac{(1-\tau)(1-z)}{8z}, 
\end{equation}
and
\begin{equation}\label{formuleernciefjvntijgvntrvijnvijnvr}
    W=z\eta V,\qquad V'=8\gamma zV,\qquad V''=64\gamma(\gamma-1)z^2V,\qquad W'=8\delta zW,\qquad W''=64\delta(\delta-1)z^2W.
\end{equation}
Combining these identities with \eqref{eq_planar_exponent_bounds} and \eqref{variederivatrebdin39gfnvitvnei}, a lengthy computation grants
\begin{equation}\label{eq_planar_numerator_identity}
 \frac{z^\gamma}{x_1\eta}\mathcal N^0u_-
 =P(z,w,\tau)+4\big[2-6\gamma+(6\gamma-1)z\big]
       \left(1-2\gamma+2\gamma z-\frac z\eta\right).
\end{equation}
First, we show that the last term in \eqref{eq_planar_numerator_identity} is non-negative. Indeed, 
\begin{equation}\label{cjenrcierjvnrtgvjnrivnrvinjrivjn}
    4\big[2-6\gamma+(6\gamma-1)z\big]\geq 0,
\end{equation}
since $z\ge 0$, and $2-6\gamma\ge 0$ and $6\gamma-1\ge 0$ by \eqref{eq_planar_exponent_bounds}. Moreover, since $2\gamma\le 1$ by \eqref{eq_planar_exponent_bounds}, then $z\eta^{-1}=z^{2\gamma}$ is concave as a function of $z$. Therefore, it is bounded above by the tangent line to its graph at $(1,1)$. Namely, \begin{equation}\label{eq_planar_concavity}
 \frac z\eta=z^{2\gamma}\leq 1-2\gamma+2\gamma z.
\end{equation}
Combining \eqref{cjenrcierjvnrtgvjnrivnrvinjrivjn} and \eqref{eq_planar_concavity},
\begin{equation}\label{eq_planar_numerator_identity_2}
 \frac{z^\gamma}{x_1\eta}\mathcal N^0u_-
 \geq P(z,w,\tau).
\end{equation}
In order to conclude the proof of \eqref{eq_planar_numerator_identity_manonevero}, we provide a lower bound for $P$. To this aim, notice that $0<z<1$. Moreover, since $1-2\gamma\ge 0$ by \eqref{eq_planar_exponent_bounds}, then $0\le \eta\le 1$. In addition, $0<\tau\leq 1$. Finally,
since $2-6\gamma\ge  0$, then $\eta^3z^{-1}=z^{2-6\gamma}\leq 1$, whence $z\geq \eta^3$. Therefore,
\[
 0\leq w\leq\eta-\eta^4
 =\eta(1-\eta)+\eta^2(1-\eta^2)\leq\frac12.
\]
Therefore, we are left to provide a positive lower bound for $P$ when $(z,w,\tau)\in[0,1]\times\left[0,\frac{1}{2}\right]\times [0,1]$. Notice that $(z,w,\tau)$ are now simply meant as three independent variables, thus regardless of their original definition. 

\smallskip
\noindent\textbf{Step 2}
We claim that 
\begin{equation}\label{eq_planar_angular_endpoints}
    P(z,w,\tau)\geq\min\{P(z,w,0),P(z,w,1)\}\qquad\text{ in }[0,1]\times\left[0,\frac{1}{2}\right]\times [0,1].
\end{equation}
In order to show \eqref{eq_planar_angular_endpoints} it suffices to prove that, for any fixed $z,w$, the polynomial $\tau\mapsto P(z,w,\tau)$ is concave. To this aim, we show that $P_{\tau\tau}\leq 0$. Indeed,
\[
\begin{aligned}
P_{\tau\tau}
&=2wD_1+2w^2C_2+6\tau w^2D_2
  +6\tau w^3C_3+12\tau^2w^3D_3\\
&=2w\big[D_1+wC_2
  +3\tau w(D_2+wC_3+2\tau wD_3)\big].
\end{aligned}
\]
Since \(w,\tau\geq0\), it suffices to show that
\begin{equation}
\label{duesdifuneficjner48eirufn}
E_1(z,w)\coloneqq D_1(z)+wC_2(z)\leq0,
\qquad
E_2(z,w,\tau)\coloneqq D_2(z)+wC_3(z)+2\tau wD_3(z)\leq0.
\end{equation}
First, \eqref{eq_planar_exponent_bounds} grants that 
\begin{equation}\label{eueduseg5688}
    D_1\leq 0,\qquad D_2\leq 0.
\end{equation}
Moreover, notice that $E_1$ can be written as
\[
E_1
=(1-2w)D_1+2w\left(D_1+\frac12C_2\right).
\]
The first term is non-positive by \eqref{eueduseg5688}. Moreover, \eqref{eq_planar_exponent_bounds} yields
\begin{equation*}
    D_1(1)+\frac12C_2(1)
=\frac{27(9\gamma-4)}8\le 0,\quad D_1'(1)+\frac12C_2'(1)
=\frac{2790\gamma-333}{16}\ge 0,\quad D_1''+\frac12C_2''
=360\gamma-126\le 0.
\end{equation*}
Therefore,
\begin{equation}\label{stimasullaltro38rjrtnrecvi}
     D_1+\frac12C_2\leq  D_1(1)+\frac12C_2(1)\le 0.
\end{equation}
Combining \eqref{eueduseg5688} and \eqref{stimasullaltro38rjrtnrecvi}, we conclude that the first inequality in \eqref{duesdifuneficjner48eirufn} holds.
For the second inequality, as above, we rewrite $E_2$ as 
\begin{equation*}
    E_2=(1-2w)D_2+2w\left(D_2+\frac{1}{2}C_3+\tau D_3\right).
\end{equation*}
The first term is non-positive by \eqref{eueduseg5688}. We deal with the second one. It can be written as 
\[
D_2+\frac12C_3+\tau D_3
=(1-\tau)\left(D_2+\frac12C_3\right)
+\tau\left(D_2+\frac12C_3+D_3\right).
\]
We therefore check its signs at \(\tau=0,1\).
By \eqref{eq_planar_exponent_bounds}, it is simple to verify that both $z$-polynomials in the previous expression are quadratic, convex polynomials, both with negative values at $0$ and at $1$. By convexity, they are negative for every $z\in[0,1]$. Therefore,  we conclude that 
\[
D_2+\frac12C_3+\tau D_3
\leq 0,
\]
whence \eqref{duesdifuneficjner48eirufn} holds.
%\[
%\begin{aligned}
%\left(-D_2-\frac12C_3\right)(0)
%&=\frac{9(661-528\gamma)}{128}>0,\\
%\left(-D_2-\frac12C_3-D_3\right)(0)
%&=\frac{9(937-1284\gamma)}{128}>0,\\
%\left(-D_2-\frac12C_3\right)(1)
%=\left(-D_2-\frac12C_3-D_3\right)(1)
%&=\frac{243(2-3\gamma)}{64}>0.
%\end{aligned}
%\]
Therefore $P_{\tau\tau}\leq 0$, and 
\eqref{eq_planar_angular_endpoints} holds.

\smallskip
\noindent\textbf{Step 3.} By \eqref{eq_planar_angular_endpoints}, we are left to estimate $P(z,w,0)$ and $P(z,w,1)$. Notice that $P(z,w,0)=a(z)$. By \eqref{eq_planar_exponent_bounds}, $a$ is concave, and moreover $a(0)\ge 0$ and $a(1)\ge 0$. Therefore $P(z,w,0)=a(z)\ge 0$. Next, we prove that $P(z,w,1)\ge 0$. In this regard, we claim that 
\begin{equation}\label{eq_planar_angular_endpoints_222222}
    P(z,w,1)\geq\min\left\{P(z,0,1),P\left(z,\frac{1}{2},1\right)\right\}\qquad\text{ in }[0,1]\times\left[0,\frac{1}{2}\right].
\end{equation}
First, notice that 
\begin{equation*}
 P(z,w,1)= b
       +w( C_1+D_1)+w^2(C_2+D_2)
       +w^3(C_3+D_3).
\end{equation*}
Hence,
\[
P_{ww}(z,w,1)=2(C_2+D_2)+6w(C_3+D_3).
\]
Recalling the definition of $\delta$ and exploiting \eqref{eq_planar_exponent_bounds},
it is not difficult to show that
\[
\begin{aligned}
C_2+D_2
&=\frac94\big[3+2\delta(1-z)\big]
  \big[9\gamma-4+6\gamma\delta(1-z)\big],\\
C_3+D_3
&=\frac{27\delta}{32}\big[3+2\delta(1-z)\big]^2.
\end{aligned}
\]
The second expression is clearly negative by \eqref{eq_planar_exponent_bounds}. For the first one, 
\eqref{eq_planar_exponent_bounds} grants that 
\[
3+2\delta(1-z)\geq3+2\delta>0,
\qquad
9\gamma-4+6\gamma\delta(1-z)\leq9\gamma-4<0.
\]
Therefore \(P_{ww}(z,w,1)<0\), and 
\eqref{eq_planar_angular_endpoints_222222} follows.

\smallskip
\noindent \textbf{Step 4.}
By \eqref{eq_planar_angular_endpoints_222222}, it suffices to estimate $P(z,0,1)$ and $P\left(z,\frac{1}{2},1\right)$.
Notice that $P(z,0,1)=b(z)$.
By \eqref{eq_planar_exponent_bounds},
\begin{equation*}
b'(z)=(97\gamma-41)+(173-448\gamma)z
  +\left(351\gamma-\frac{297}{2}\right)z^2<-10+36z(1-z)\leq-1.
\end{equation*}
Therefore
\[
P(z,0,1)=b(z)\geq b(1)=4(4\gamma-1)>0.
\]
Finally, we estimate $P\left(z,\frac{1}{2},1\right)$. Since
\[
P\left(z,\frac12,1\right)
=b(z)+\frac12(C_1+D_1)
+\frac14(C_2+D_2)+\frac18(C_3+D_3),
\]
then $P\left(z,\frac{1}{2},1\right)$ has the explicit expression
\[
\frac1{512}\big[
42970\gamma-12999
 +(20264-69652\gamma)z+(8531-1972\gamma)z^2
 +(59904\gamma-25344)z^3
\big].
\]
We claim that 
\(P\left(z,\frac{1}{2},1\right)\geq P\left(1,\frac12,1\right)\).
Notice that 
\[
P\left(1,\frac12,1\right)
=\frac{31250\gamma-9548}{512}.
\]
Therefore,
\[
P\left(z,\frac12,1\right)-P\left(1,\frac12,1\right)
=\frac{1-z}{512}\big[
11720\gamma-3451
+(16813-57932\gamma)z
+(25344-59904\gamma)z^2
\big].
\]
By \eqref{eq_planar_exponent_bounds},
\[
\frac{11720\gamma-3451}{512}>\frac14,
\qquad
\frac{16813-57932\gamma}{512}>-3,
\qquad
\frac{25344-59904\gamma}{512}>12.
\]
Therefore,
\[
P\left(z,\frac12,1\right)-P\left(1,\frac12,1\right)
\geq(1-z)\left(\frac14-3z+12z^2\right)
\geq0.
\]
In conclusion,
\[
P\left(z,\frac12,1\right)
\geq P\left(1,\frac12,1\right)
=\frac{31250\gamma-9548}{512}
>0,
\]
where the last inequality follows again from
\eqref{eq_planar_exponent_bounds}. Combining this information with \eqref{eq_planar_angular_endpoints} and \eqref{eq_planar_angular_endpoints_222222}, we conclude that $P\ge 0$. In view of \eqref{eq_planar_numerator_identity_2}, \eqref{eq_planar_numerator_identity_manonevero} follows.

\smallskip
\noindent\textbf{Step 5.} In order to conclude the proof of the differential inequality in \eqref{subsoluzionejernfvojnttijvnrtvjnr}, it suffices to show that $\Delta u_-\geq0$ in view of \eqref{splittingdeglioperatioriermerijn} and \eqref{eq_planar_numerator_identity_manonevero}.
To this aim, by the above computations,
\[
\frac{\Delta u_-}{x_1}
=8V'+4|x|^2V''+18W+48x_1^2W'+12x_1^2|x|^2W''.
\]
Therefore, by \eqref{duecoseederftrgrrgt} and \eqref{formuleernciefjvntijgvntrvijnvijnvr}
\begin{equation*}
    \begin{split}
        \frac{\Delta u_-}{x_1z^{1-\gamma}}
&=64\gamma+32\gamma(\gamma-1)(1-z)+\eta\big[18+48\delta\tau(1-z)
          +12\delta(\delta-1)\tau(1-z)^2\big]\\
&=32\gamma\big[2+(\gamma-1)(1-z)\big]+\eta\big(18+12\delta\tau(1-z)
                  [4+(\delta-1)(1-z)]\big).
    \end{split}
\end{equation*}
Since \(0\leq1-z\leq1\) and \(\delta\le1\), we have
\[
4+(\delta-1)(1-z)\geq3+\delta>0,
\]
 so that, since $\delta\le 0$ and $\tau\leq 1$,
 \begin{equation*}
    \begin{split}
        \frac{\Delta u_-}{x_1z^{1-\gamma}}
\geq 32\gamma\big[2+(\gamma-1)(1-z)\big]+\eta\big(18+12\delta(1-z)
                  [4+(\delta-1)(1-z)]\big).
    \end{split}
\end{equation*}
Since the above expression is affine in $\eta$ it is bounded below by the minimum between its values at $\eta=0$ and $\eta=1$.
First, at $\eta=0$,  \eqref{eq_planar_exponent_bounds} yelds
\[
32\gamma\big[2+(\gamma-1)(1-z)\big]
\geq32\gamma(1+\gamma)
=8+16\gamma>12\ge 0.
\]
Moreover, at $\eta=1$,
\[
32\gamma\big[2+(\gamma-1)(1-z)\big]
 +18+12\delta(1-z)[4+(\delta-1)(1-z)]=
29-74\gamma+(372\gamma-110)z
+12\delta(\delta-1)z^2.
\]
Again by \eqref{eq_planar_exponent_bounds}, all the three coefficients are non-negative. Therefore, since $0\le z\le 1$, we conclude that $\Delta u_-\geq 0.$ 
 Therefore, the differential inequality in \eqref{subsoluzionejernfvojnttijvnrtvjnr} follows.
Finally, since $2\gamma-2<-1$,
\[
 \frac{u_-}{x_1(1+8|x|^2)^\gamma}
 =1+3x_1^2(1+8|x|^2)^{2\gamma-2}
 \leq1+\frac{3|x|^2}{1+8|x|^2}<\frac32.
\]
Since $(1+8|x|^2)^\gamma\leq8^\gamma(1+|x|^2)^\gamma
<2(1+|x|^2)^\gamma$, this proves the upper bound in \eqref{subsoluzionejernfvojnttijvnrtvjnr}. The thesis follows.
\end{proof}

We exploit the above barriers to prove \Cref{propunistimamain3948573883derivata}.
\begin{proof}[Proof of \Cref{propunistimamain3948573883derivata}]
Let $u$ be a solution to \eqref{eq_diri_ball_eps}, with $\eps\in(0,1]$, $R\ge 0$ and $K=K(R)$ as in \eqref{definizionekappaperheightestimates3948568}. Let $u_-$ be the lower barrier of \Cref{prop:barrier-lower-higher} or
\Cref{prop:barrier-lower-planar}, according to wether $n\geq 3$ or $n=2$. In both cases,
\begin{equation}\label{barierabassaprontissima}
    x_1\leq u_-(x)\leq C_nx_1(1+|x|^2)^\gamma
 \qquad\text{in }\{x_1\geq 0\},\qquad \Ne u_-\geq0\qquad\text{in }\{x_1\ge 0\}, 
\end{equation}
where $C_n$ is as in \eqref{defigammaecn}.
Let $U$ be as in \Cref{prop:barrier-upper}, and set
\[
 \lambda\coloneqq(4^\gamma C_n)^{\frac{1}{1-2\gamma}}\geq1,
 \qquad u_+(x)\coloneqq\lambda^2U\left(\frac x\lambda\right).
\]
Since
\begin{equation*}
    Du_+(x)=\lambda DU\left(\frac x\lambda\right),\qquad D^2u_+(x)=D^2U\left(\frac x\lambda\right),
\end{equation*} 
then
\begin{equation}\label{eq:barrier-dilation}
 \Ne u_+(x)=\lambda^2
       \left(\mathcal N^{\frac\eps\lambda}U\right)\left(\frac x\lambda\right)\leq0
       \qquad\text{on }\{x_1\ge 0\}
\end{equation}
By \Cref{prop:barrier-upper}, because $0\le \eps\lambda^{-1}\leq 1$.
Moreover, by
\eqref{eq:barrier-upper} and by definition fo $\lambda$,
\begin{equation}\label{barieraaltaprontissima}
    C_nx_1(1+|x|^2)^\gamma
 \leq u_+(x)\leq2\lambda x_1(1+|x|^2)^\gamma
 \qquad\text{in }\{x_1\geq 0\}.
\end{equation}
Set $\Om=B_R\cap \{x_1\ge 0\}$. Notice that $\partial\Om=\left(\overline B_R\cap\{x_1=0\}\right)\cup\left(\partial B_R\cap \{x_1\ge 0\}\right)$. The choice of $K$, combined with \eqref{barierabassaprontissima}, \eqref{barieraaltaprontissima} and the oddness of $u$ in $x_1$, yields
\[
 u_-=0=u=u_+\text{ on }\overline B_R\cap\{x_1=0\},\qquad u_-\leq C_n(1+R^2)^\gamma x_1=K(R)x_1=u\leq u_+\text{ on }\partial B_R\cap \{x_1\ge 0\}.
\]
Therefore,
\begin{equation*}
    \Ne u_-\geq\Ne u\geq \Ne u_+\quad\text{on }\Om,\qquad u_-\leq u\leq u_+\quad\text{on }\partial\Om.
\end{equation*}
By
\Cref{prop_comparison_principle},  we conclude that $u_-\leq u\leq u_+ \text{ on }\Om,$ whence \eqref{uniformestimatesproprincheigh} follows by \eqref{barierabassaprontissima} and  \eqref{barieraaltaprontissima}.
\end{proof}

\subsection{Uniform interior gradient estimates}
In this section, we provide interior gradient estimates for solutions to $\Qe u=0$, uniformly with respect to $\eps\in(0,1]$, on domains which are detached from the origin. Our approach follows closely that of \cite[Theorem 1.2]{MR5029815}. The main difference consists of the uniform control on the Riemannian curvature terms provided by \eqref{riemcurvaturetermspresoda3949596}.
\begin{proposition}\label{thm_uniform_gradient}
Let $0\le\varrho_1\le\varrho_2\le R$.   Let $u\in C^\infty(\overline B_R)$ satisfy $\Qe u=0$ on $B_R$.
Let $r,\gamma_1\ge 0$ and $\hat x\in B_R$ be such that
\begin{equation*}\label{gamma_1_bound}
 B_r(\hat x)\subseteq \{\varrho_1\le|x|\le\varrho_2\},\qquad |u|\leq\gamma_1\text{ on $\overline { B_r(\hat x)}$}.
\end{equation*}
Then there exists $C=C(n,\gamma_1,\varrho_1,\varrho_2,r)$, whence independent of $\eps\in(0,1]$ and $R\ge 0$, such that 
\begin{equation}\label{eq_ug_estimate}
 |Du(\hat x)|\leq C(n,\gamma_1,\varrho_1,\varrho_2,r).
\end{equation}
\end{proposition}
\begin{proof}
   Let $S$ be the $t$-graph of $\langle x,y\rangle+u(x)$ in $\hh^n$ over $B_R\times\rr^n$. Since $u$ solves $\Qe u=0$, $H^\eps\equiv 0$ by \eqref{espressionecurvaturamediariempergarficispeciali}. 
   
   \smallskip
   \noindent \textbf{Step 1.}
   Define $\varphi\in C^\infty(\overline{B_r(\hat x)})$ by 
   \begin{equation*}
 \varphi(x)\coloneqq u(\hat x)-1+
 \bigl(\gamma_1-u(\hat x)+1\bigr)\frac{|x-\hat x|^2}{r^2}.
\end{equation*}
By construction, $(u-\varphi)(\hat x)=1$ and $u-\varphi\leq 0$ on $\partial B_r(\hat x)$. Let $\gamma_2=\gamma_2(\gamma_1,r)\ge 0$ be such that
\begin{equation*}
    |\varphi|\leq \gamma_2,\qquad |D\varphi|\leq \gamma_2,\qquad \langle D^2\varphi\,v,v\rangle\leq\gamma_2|v|^2\text{ for any }v\in\rr^n,\qquad \text{on }\overline{B_r(\hat x)}.
\end{equation*}
Fix $K\ge 0$ to be specified later, and set 
\begin{equation*}
 \eta(s)\coloneqq (e^{Ks}-1)e^{-K(\gamma_1+\gamma_2)},\qquad s\in\rr.
\end{equation*}
By definition of $\gamma_1$ and $\gamma_2,$ then 
\begin{equation}\label{cosedieta3848685}
    0\leq \eta\big((u-\varphi)_+\big)\leq 1\quad\text{on }\overline{B_r(\hat x)},\qquad  \eta\big((u-\varphi)_+\big)=0\quad\text{on }\partial B_r(\hat x).
\end{equation}
Define
\begin{equation*}
 \Phi\coloneqq
     \eta((u-\varphi)_+)W_\varepsilon\quad\text{on }\overline{B_r(\hat x)},
 \qquad M\coloneqq \max_{\overline{B_r(\hat x)}}\Phi,
\end{equation*}
where $W_{\eps}$ is defined in \eqref{eq:Weps}.
By \eqref{cosedieta3848685}, $\Phi=0$ on $\partial B_r(\hat x)$. Therefore, since $(u-\varphi)(\hat x)=1$, then $\Phi(\hat x)\ge 0$, the maximum $M$ is achieved at some $\bar x\in B_r(\hat x)$. In particular, $(u-\varphi)(\bar x)\ge 0$ at $\bar x$, whence $\Phi=
     \eta(u-\varphi)W_\varepsilon\text{ locally near }\bar x$.
Recalling the identification of \Cref{sec_tgraph}, define $\Psi\in C^\infty(S)$ by
\begin{equation*}
    \Psi(x,y)=\Psi(x)\coloneqq  \eta(u-\varphi)-M\We^{-1}.
\end{equation*}
By construction, for any $y\in\rr^n$, $(\bar x,y)$ is a local maximum point of $\Psi$, whence
\begin{equation}\label{condizionimassimoinuge3485886}
    \nabla^{\eps,S}\Psi(\bar x)=0,\qquad \Delta^{\eps,S}\Psi(\bar x)\leq 0.
\end{equation}
Therefore, since
\begin{equation}\label{esprimereweconalpha38486848}
    \We^{-1}\overset{\eqref{normaleriemannianaepsilon},\eqref{normalipergraficiridotti38586}}{=}\frac{\alpha}{\sqrt{1+\eps^2\alpha^2}},
\end{equation}
then, at $\bar x$,
\begin{equation}\label{laplacianouno384959649}
    \begin{split}
        0&\geq \Delta^{\eps,S}\Psi\\
        &=\eta''(u-\varphi)|\nabla^{\eps,S}(u-\varphi)|^2+\eta'(u-\varphi)\Delta^{\eps,S}(u-\varphi)-M\Delta^{\eps,S}\big(\We^{-1})\\
        \overset{\eqref{jacobitipeequation384969}}&{=}\eta''(u-\varphi)|\nabla^{\eps,S}(u-\varphi)|^2+\eta'(u-\varphi)\Delta^{\eps,S}(u-\varphi)+M\We^{-1}\Big(\left|h^\eps\right|^2+\ric^\eps(\n^\eps,\n^\eps)\Big).
    \end{split}
\end{equation}
\textbf{Step 2.}
We claim (without intention to be sharp) that 
\begin{equation}\label{boundpermuniformeo48576884}
    M\leq M_0\coloneqq \sqrt{2+8\varrho_2^2+4\gamma_2^2}.
\end{equation}
Suppose not. We show that appropriate lower bounds in \eqref{laplacianouno384959649} and a suitable choice of $K$ lead to a contradiction. First, we bound the gradient term. Since $\eta\big((u-\varphi)(\bar x)\big)\leq 1$ by \eqref{cosedieta3848685} and $\Psi(\bar x)=0$ by construction, the choice of $\bar x$ yields $ M_0\le M\le \We$. Therefore
\begin{equation}\label{aux1uge3848697}
4|D(u-\varphi)|^2\geq 2|Du|^2-4|D\varphi|^2\geq
   |Du|^2+2+8\varrho_2^2+4\gamma_2^2-4|x|^2-\eps^2-4|D\varphi|^2\geq |Du|^2\qquad\text{at }\bar x
\end{equation}
and
\begin{equation}\label{auxdueuge83496974}
      |Du|^2\geq 2+8\varrho_2^2+4\gamma_2^2-4|x|^2-\eps^2\geq 4\varrho_2^2+1\geq 4|x|^2+\eps^2\qquad\text{at }\bar x,
\end{equation}
whence, at $\bar x$,
\begin{equation}\label{sitmagradumenophiuge34869473}
    \begin{split}
        \left|\nabla^{\eps,S}(u-\varphi)\right|^2\overset{\eqref{gradientesulgraficopercosenony38486}}&{=}\frac{(|Du|^2+4|x|^2+\eps^2)|D(u-\varphi)|^2-\left\langle D(u-\varphi),Du\right\rangle^2}{\We^2}\\
        &\geq \left(\frac{4|x|^2+\eps^2}{\We^2}\right)|D(u-\varphi)|^2\\
        \overset{\eqref{aux1uge3848697}}&{\geq}\left(\frac{|Du|^2}{\We^2}\right)\left(\frac{4|x|^2+\eps^2}{4}\right)\\
        \overset{\eqref{auxdueuge83496974}}&{\geq}\frac{4|x|^2+\eps^2}{8}\\
        &\geq \frac{\varrho_1^2}{2}.
    \end{split}
\end{equation}
\noindent \textbf{Step 3.} Next, we bound the second-order term. To this aim, observe that
\begin{equation*}
    g^{ij}_\eps(x,y)\overset{\eqref{formaspecialecoso2939469}}{=}\delta_{ij}-\frac{u_i(x)u_j(x)}{\We(x)^2}=\We(x)^{-2}\Ae_{ij}(x,Du(x)),\qquad i,j=1,\ldots, n.
\end{equation*}
Therefore, since $\Qe u=0$, 
\begin{equation*}
        \Delta^{\eps,S}u\overset{\eqref{equazioneperlaplacianograficoridotto38486854}}{=}\frac{1}{\We^2}\sum_{i,j=1}^{n}\Ae_{ij} D^2u_{ij}+\frac{4\left\langle Du,x\right\rangle}{\We^2}=\frac{8\left\langle Du,x\right\rangle}{\We^2}\geq -2.
\end{equation*}
Moreover, by definition of $\gamma_2$,
\begin{equation*}
    \Delta^{\eps,S}\varphi \overset{\eqref{equazioneperlaplacianograficoridotto38486854}}{=}\Delta\varphi-\frac{\left\langle D^2 \varphi\,Du,Du\right\rangle}{\We^2}+\frac{4\left\langle D\varphi,x\right\rangle}{\We^2}\leq (n+1)\gamma_2+\frac{\gamma_2}{|x|}\leq (n+1)\gamma_2+\frac{\gamma_2}{\varrho_1}.
\end{equation*}
Therefore,
\begin{equation}\label{stimalapluge3784869847374}
    \Delta^{\eps,S}(u-\varphi)\geq -2-\left(n+1+\frac{1}{\varrho_1}\right)\gamma_2.
    \end{equation}
    \noindent \textbf{Step 4.} Finally, we bound the curvature term. Notice that $S$ is non-characteristic over $\{\varrho_1\le |x|\le \varrho_2\}\times\rr^n$. The following quantities are tacitly evaluated at $\bar x$. Since
    \begin{equation}\label{chainrulealphawe3885764}
    \alpha \We J(\vh)\left(\We^{-1}\right)\overset{\eqref{esprimereweconalpha38486848}}{=}\frac{J(\vh)\alpha}{1+\eps^2\alpha^2},
    \end{equation}
    we discard the non-negative terms in \eqref{riemcurvaturetermspresoda3949596}, obtaining 
    \begin{equation*}
        M\We^{-1}\Big(\left|h^\eps\right|^2+\ric^\eps(\n^\eps,\n^\eps)\Big)\geq 4M\We^{-1}\left(\frac{J(\vh)\alpha}{1+\eps^2\alpha^2}\right)\overset{\eqref{chainrulealphawe3885764}}{=}4M\alpha J(\vh)\left(\We^{-1}\right).
    \end{equation*}
    Noticing that 
    \begin{equation*}
        0\overset{\eqref{condizionimassimoinuge3485886}}{=}J(\vh)\Psi=\eta'(u-\varphi)J(\vh)(u-\varphi)-MJ(\vh)\left(\We^{-1}\right),
    \end{equation*}
    then 
    \begin{equation*}
        M\We^{-1}\Big(\left|h^\eps\right|^2+\ric^\eps(\n^\eps,\n^\eps)\Big)\geq \eta'(u-\varphi)\big(4\alpha J(\vh)(u-\varphi)\big).
    \end{equation*}
    By \eqref{normalipergraficiridotti38586} and \eqref{espressionecurvaturamediariempergarficispeciali}, and arguing as in Step 3,
    \begin{equation*}
        4\alpha J(\vh)(u-\varphi)=\frac{8\left\langle Du,x\right\rangle}{\W^2}-\frac{8\left\langle D\varphi,x\right\rangle}{\W^2}\geq -2-\frac{2\gamma_2}{\varrho_1},
    \end{equation*}
    In conclusion, and noticing that $\eta'\geq 0$, 
    \begin{equation}\label{outputofstep4uge3845764}
        M\We^{-1}\Big(\left|h^\eps\right|^2+\ric^\eps(\n^\eps,\n^\eps)\Big)\geq \eta'(u-\varphi)\left(-2-\frac{2\gamma_2}{\varrho_1}\right).
    \end{equation}
    \noindent \textbf{Step 5.} 
    Choose
    \begin{equation*}
        K=\frac{2}{\varrho_1^2}\left(5+\left(n+1+\frac{3}{\varrho_1}\right)\gamma_2\right).
    \end{equation*}
    Since $\eta',\eta''\geq 0$ and $\eta''=K\eta'$, we combine \eqref{laplacianouno384959649} with \eqref{sitmagradumenophiuge34869473}, \eqref{stimalapluge3784869847374} and \eqref{outputofstep4uge3845764}. Therefore, at $\bar x$, 
    \begin{equation*}
    \begin{split}
          0&\geq \eta''(u-\varphi)|\nabla^{\eps,S}(u-\varphi)|^2+\eta'(u-\varphi)\Delta^{\eps,S}(u-\varphi)+M\We^{-1}\Big(\left|h^\eps\right|^2+\ric^\eps(\n^\eps,\n^\eps)\Big)\\
          &\geq  \eta''(u-\varphi)\left(\frac{\varrho_1^2}{2}\right)+\eta'(u-\varphi)\left(-4-\left(n+1+\frac{3}{\varrho_1}\right)\gamma_2\right)\\
          &=\eta'(u-\varphi)\left(\frac{K\varrho_1^2}{2}-4-\left(n+1+\frac{3}{\varrho_1}\right)\gamma_2\right)\\
          &=\eta'(u-\varphi)\\
          &\ge 0,
    \end{split}
    \end{equation*}
    a contradiction. Hence \eqref{boundpermuniformeo48576884} holds. By definition of $M$, recalling that $(u-\varphi)(\hat x)=1$, we conclude that
    \begin{equation*}
        |Du(\hat x)|\leq \We(\hat x)=\frac{\Phi(\hat x)}{\eta(1)}\leq \frac{M}{\eta(1)}\overset{\eqref{boundpermuniformeo48576884}}{\leq}\frac{M_0}{\eta(1)}.
    \end{equation*}
    The thesis follows. 
\end{proof}
\subsection{Uniform positive lower bounds on the first partial derivative}\label{fvrtg5ytbgtygvtr}
Throughout this section, we fix $\eps\in(0,1]$, $R\ge 0$, $K=K(R)$ as in \eqref{definizionekappaperheightestimates3948568}, and the solution $u\in C^\infty(\overline B_R)$ given by \eqref{eq_diri_ball_eps}. We provide the following positive lower bound for $u_1$. 
\begin{proposition}\label{prooffromboundtoboundderprop3848}
    It holds that 
    \begin{equation}\label{ugequnoonbrproop}
        u_1\geq 1 \qquad\text{on } B_R.
    \end{equation}
\end{proposition}
We prove  \Cref{prooffromboundtoboundderprop3848} in several steps, combining \Cref{propunistimamain3948573883derivata} with several applications of elementary maximum-principle-related techniques for the linearized operator associated with $\Ne$. Accordingly, set
\begin{equation}\label{linearizeddef3845}
    \Le\varphi\coloneqq \sum_{i,j=1}^n\Ae(x,Du)_{ij}D^2\varphi_{ij}+\left\langle 2\Delta u\,Du-2D^2 u\,Du-4x,D\varphi\right\rangle.
\end{equation}
Throughout this section, we write  $\Ae=\Ae(x,Du)$. By \Cref{sec_reg_eq}, $\Le$ is a linear, second-order, uniformly elliptic operator over $ B_R$. In the following, we denote by $e_1,\ldots, e_n$ the canonical basis of $\rr^n$.  At a fixed point $x\in \overline B_R$, we adopt the notations
\begin{equation*}
    r=r(x)\coloneqq \sqrt{x_2^2+\cdots+x_n^2},\qquad e_r=e_r(x)\coloneqq\frac{1}{r}\sum_{i=2}^nx_i e_i,\qquad u_r\coloneqq \langle Du,e_r\rangle =\frac{1}{r}\sum_{i=2}^nx_i u_i.
\end{equation*}
Notice that $e_r$ and $u_r$ are defined only when $r\neq 0$. Nevertheless, by invariance with respect to $x_2,\ldots,x_n$, $u_r$ extends continuously to $0$ on $r=0$. Accordingly, with abuse of notation we will write $u_r=0$ on when $r=0$. When instead $r\ge 0$, since $u$ is invariant under orthogonal transformations of $x_2,\ldots,x_n$,
\begin{equation}\label{gradienteinroti394969}
    Du(x)=u_1(x) e_1+u_r(x)e_r(x).
\end{equation}
In particular, for any fixed $\tilde x$ such that $r(\tilde x)\ge 0$, there are constant vector fields $\tau_3,\ldots,\tau_n$ such that $e_1,e_r(\tilde x),\tau_3,\ldots,\tau_n$ is a orthonormal basis at $\tilde x$ and 
\begin{equation}\label{campiortogonali384956995}
    \langle Du(\tilde x),\tau_i\rangle=0\qquad i=3,\ldots,n. 
\end{equation}
Moreover, if $i=3,\ldots,n$ is fixed,  \eqref{gradienteinroti394969} and the fact that $\langle e_1,\tau_i\rangle =0$ yield
\begin{equation*}
    \langle Du(x),\tau_i\rangle =u_r(x)\langle e_r(x),\tau_i\rangle=u_r(x)\frac{\langle x,\tau_i\rangle}{r(x)}\qquad \text{locally near $\tilde x$}.
\end{equation*}
Differentiating the previous identity, recalling that $\tau_i$ is constant, gives
\begin{equation}\label{eigenvectorsipreliminaries394596}
    D^2u(x)\,\tau_i =D\left(\frac{u_r(x)}{r(x)}\right) \langle x,\tau_i\rangle+\frac{u_r(x)}{r(x)}\tau_i\quad \text{locally near $\tilde x$},\qquad  D^2 u(\tilde x)\,\tau _i=\frac{u_r(\tilde x)}{r(\tilde x)}\tau_i.
\end{equation}
In particular, $\tau_3,\ldots,\tau_n$ are eigenvectors of $D^2 u(\tilde x)$, with common eigenvalue $r(\tilde x)^{-1}u_r(\tilde x)$. We first collect some basic features of $\Le$. 
\begin{proposition}
    The following facts hold on $B_R$.
    \begin{align}
    \Le u&=8\langle x, Du\rangle -2(4|x|^2+\eps^2)\,\Delta u,\label{prole1}\\
    \Le x_i&=2 u_i\Delta u-2(D^2 u\,Du)_i-4x_i,\qquad i=1,\ldots,n,\label{prole2}\\
    \Le u_i&=4u_i-8x_i\,\Delta u,\label{prole3}\\
    \Le(\varphi\psi)&=\varphi\Le \psi+\psi\Le\varphi+2\sum_{i,j=1}^n\Ae_{ij}\varphi_i\psi_j,\label{prole4}\\
   \Le\langle x,Du\rangle&=16\langle x,Du\rangle-(16|x|^2+2\varepsilon^2)\Delta u.\label{prole5},\\
\Le\left(\frac{x_i}{r}\right)
&=-(n-2)W_\varepsilon^2\frac{x_i}{r^3}
\qquad r>0,\,i=2,\ldots,n,\label{prole6}\\
\Le u_r
&=
4u_r-8r\Delta u
+(n-2)W_\varepsilon^2\frac{u_r}{r^2},\qquad r>0.\label{prole7}
    \end{align}
\end{proposition}
\begin{proof}
    The proof of \eqref{prole1} is a simple consequence of $\Ne u=0$. The proof of \eqref{prole2} follows directly by \eqref{linearizeddef3845}. We prove \eqref{prole3}. To this aim, recalling that $\Ne u=0,$
    \begin{equation*}
        \begin{split}
            0&=\partial_i\left(\Ne u\right)\\
            &=\sum_{j,k=1}^n\partial_{x_i}\Ae_{jk}u_{jk}+\sum_{j,k,l=1}^n\partial_{p_l}\Ae_{jk}u_{il}u_{jk}+\sum_{j,k=1}^n\Ae_{jk}(u_i)_{jk}-4 u_i-4\sum_{j=1}^n x_j u_{ij}\\
            \overset{\eqref{derivatematriceae384869}}&{=}8x_i\Delta u+2\langle Du,Du_i\rangle \Delta u-2\langle D^2 u\,Du, Du_i\rangle +\sum_{j,k=1}^n\Ae_{jk}(u_i)_{jk}-4 u_i-4\langle x, Du_i\rangle\\
            &=\Le u_i-4 u_i+8x_i\Delta u.
        \end{split}
    \end{equation*}
    This proves \eqref{prole3}. In addition, \eqref{prole4} follows directly by linearity of $\Le$ and the product rule for derivatives. Next, we prove \eqref{prole5}. Accordingly,
    \begin{equation*}
        \begin{split}
            \Le\langle x,Du\rangle &=\sum_{i=1}^n\Le (x_i u_i)\\
            \overset{\eqref{prole4}}&{=}\sum_{i=1}^nx_i\Le ( u_i)+\sum_{i=1}^nu_i\Le ( x_i)+2\sum_{i,j,k=1}^n \Ae_{jk}\delta_{ij}u_{ik}\\
            \overset{\eqref{prole2},\eqref{prole3}}&{=}4\langle x,Du\rangle -8|x|^2\Delta u+2|Du|^2\Delta u -2\langle D^2 u\,Du,Du\rangle -4\langle x,Du\rangle+2\sum_{i,k=1}^n \Ae_{ik}u_{ik}\\
            \overset{\eqref{eq_diri_ball_eps}}&{=}16\langle x,Du\rangle-(16|x|^2+2\varepsilon^2)\Delta u.
        \end{split}
    \end{equation*}
    We prove \eqref{prole6}.
    Fix $i=2,\ldots,n$. For $j,k=2,\ldots,n$, a direct computation gives
\begin{equation}\label{deriprimedixisur385896}
    \partial_j\left(\frac{x_i}{r}\right)
=\frac{\delta_{ij}}{r}-\frac{x_i x_j}{r^3}
\end{equation}
and
\[
\partial_{jk}\left(\frac{x_i}{r}\right)
=
-\frac{
\delta_{ij}x_k+\delta_{ik}x_j+\delta_{jk}x_i
}{r^3}
+\frac{3x_i x_j x_k}{r^5}.
\]
All first and second derivatives involving differentiation
with respect to $x_1$ vanish. Consequently,
\[
\sum_{j=1}^n
x_j\partial_j\left(\frac{x_i}{r}\right)
=
\frac{x_i}{r}-\frac{x_i r^2}{r^3}
=0,
\]
\[
\Delta\left(\frac{x_i}{r}\right)
=
-\frac{(n+1)x_i}{r^3}
+\frac{3x_i r^2}{r^5}
=
-\frac{(n-2)x_i}{r^3},
\]
and
\[
\sum_{j,k=2}^n
x_j x_k\partial_{jk}\left(\frac{x_i}{r}\right)
=
-\frac{3x_i r^2}{r^3}
+\frac{3x_i r^4}{r^5}
=0.
\]
Hence, recalling that \eqref{gradienteinroti394969} implies
\begin{equation}\label{derisoluorvijtrvntrijn48495966}
    u_j=\frac{x_j\,u_r}{r},
\qquad u_{jk}=\partial_k\left(\frac{x_j}{r}\right)u_r+\frac{x_j}{r}\partial _k u_r,\qquad  j,k=2,\ldots,n ,
\end{equation}
then
\begin{equation*}
\sum_{j,k=1}^n
\Ae_{jk}
\partial_{jk}\left(\frac{x_i}{r}\right)
=
\We^2\Delta\left(\frac{x_i}{r}\right)
-\frac{u_r^2}{r^2}
\sum_{j,k=2}^n
x_j x_k
\partial_{jk}\left(\frac{x_i}{r}\right)=-(n-2)W_\varepsilon^2\frac{x_i}{r^3}.
\end{equation*}
Moreover, recalling again \eqref{gradienteinroti394969},
\begin{equation}\label{384969548375685}
2\sum_{k=1}^n u_k u_{jk}
=\partial_j(|Du|^2)
=\frac{x_j}{r}\,\partial_r(|Du|^2).
\end{equation}
Therefore, combining \eqref{derisoluorvijtrvntrijn48495966} and \eqref{384969548375685},
\[
2(\Delta u)u_j
-2\sum_{k=1}^n u_k u_{jk}
-4x_j
=
\frac{x_j}{r}
\left(
2(\Delta u)u_r-\partial_r(|Du|^2)-4r
\right).
\]
In conclusion,
\[
\left\langle 2\Delta u\,Du-2D^2 u\,Du-4x,D\left(\frac{x_i}{r}\right)\right\rangle=
\left(\frac{2(\Delta u)u_r-\partial_r(|Du|^2)-4r}{r}
\right)\sum_{j=2}^n
x_j\partial_j\left(\frac{x_i}{r}\right)
=0.
\]
Finally, we prove \eqref{prole7}. First,
\begin{equation*}
\Le u_r\overset{\eqref{prole4}}{=}
\sum_{i=2}^n\frac{x_i}{r}\Le u_i
+\sum_{i=2}^n u_i\Le\left(\frac{x_i}{r}\right)
+
2\sum_{i=2}^n\sum_{j,k=1}^n
\Ae_{jk}
\partial_j\left(\frac{x_i}{r}\right)u_{ik}.
\end{equation*}
By \eqref{prole3} and \eqref{prole6},
\begin{align*}
\sum_{i=2}^n\frac{x_i}{r}\Le u_i
&=
\sum_{i=2}^n\frac{x_i}{r}
\left(4u_i-8x_i\Delta u\right)
=4u_r-8r\Delta u,\\
\sum_{i=2}^n u_i\Le\left(\frac{x_i}{r}\right)
&=
-\frac{(n-2)W_\varepsilon^2}{r^3}
\sum_{i=2}^n x_i u_i
=-(n-2)W_\varepsilon^2\frac{u_r}{r^2}.
\end{align*}
Moreover, for $i=2,\ldots,n$, \eqref{deriprimedixisur385896} and \eqref{derisoluorvijtrvntrijn48495966} yield
\begin{equation*}
\sum_{j=1}^n u_j
\partial_j\left(\frac{x_i}{r}\right)
=
u_r\sum_{j=2}^n\frac{x_j}{r}
\left(\frac{\delta_{ij}}r-\frac{x_i x_j}{r^3}\right)=
u_r\left(
\frac{x_i}{r^2}
-\frac{x_i}{r^4}\sum_{j=2}^n x_j^2
\right)
=0.
\end{equation*}
Therefore,
\begin{equation}\label{partialuno3845969679796}
2\sum_{i=2}^n\sum_{j,k=1}^n
\Ae_{jk}
\partial_j\left(\frac{x_i}{r}\right)u_{ik}=
2W_\varepsilon^2
\sum_{i,j=2}^n
\partial_j\left(\frac{x_i}{r}\right)u_{ij}.
\end{equation}
Since
\[
\sum_{i=2}^n\left(\frac{x_i}{r}\right)^2=1,
\]
then
\[
\sum_{i=2}^n
\frac{x_i}{r}\,
\partial_j\left(\frac{x_i}{r}\right)
=0.
\]
Therefore
\begin{equation}\label{partialdue39549968483}
\sum_{i,j=2}^n
\partial_j\left(\frac{x_i}{r}\right)u_{ij}
=
u_r\sum_{i,j=2}^n
\left(
\partial_j\left(\frac{x_i}{r}\right)
\right)^2
+
\sum_{j=2}^n\partial_j u_r
\sum_{i=2}^n
\frac{x_i}{r}\,
\partial_j\left(\frac{x_i}{r}\right)
=
u_r\sum_{i,j=2}^n
\left(
\partial_j\left(\frac{x_i}{r}\right)
\right)^2.
\end{equation}
Therefore, noticing that 
\begin{equation*}
\sum_{i,j=2}^n
\left(
\partial_j\left(\frac{x_i}{r}\right)
\right)^2=
\sum_{i,j=2}^n
\left(\frac{\delta_{ij}}r-\frac{x_i x_j}{r^3}\right)^2=
\frac{n-1}{r^2}
-\frac{2}{r^4}\sum_{i=2}^n x_i^2
+\frac1{r^6}
\sum_{i,j=2}^n x_i^2x_j^2=\frac{n-2}{r^2},
\end{equation*}
we conclude by \eqref{partialuno3845969679796} and \eqref{partialdue39549968483} that 
\[
2\sum_{i=2}^n\sum_{j,k=1}^n
\Ae_{jk}
\partial_j\left(\frac{x_i}{r}\right)u_{ik}
=
2(n-2)W_\varepsilon^2\frac{u_r}{r^2}.
\]
Combining the above computations, \eqref{prole7} follows.
\end{proof}
The following basic properties of $u$ follow easily by its construction.
\begin{proposition}\label{prop:mono-boundary}
It holds that 
\begin{equation}\label{eq:mono-basic}
    u\leq Kx_1\qquad\text{in }B_R\cap\{x_1>0\}.
\end{equation}
\begin{equation}\label{eq:mono-boundary}
 Du=Ke_1+d(x)\frac{x}{R}\quad \text{on }\partial B_R,\qquad \text{where }d(x)\coloneqq \left\langle D(u-Kx_1),\frac{x}{R}\right\rangle\geq 0\quad\text{on } \partial B_R\cap \{x_1\ge 0\}.
\end{equation}
\end{proposition}
\begin{proof}
First, notice that 
\begin{equation*}
    \Ne Kx_1=-4Kx_1\le 0\qquad\text{on }B_R\cap\{x_1\ge 0\},\qquad u\equiv K x_1\qquad\text{on }\partial \big(B_R\cap\{x_1\ge 0\}\big).
\end{equation*}
The last identity follows by \eqref{eq_diri_ball_eps} and because $u(0,x_2,\ldots,x_n)=0$ by oddness. Therefore, the comparison principle \cite[Theorem 3.3]{MR1814364} (cf. \eqref{linear_elliptic_forreg} and the subsequent observation) yields \eqref{eq:mono-basic}. Next, we prove \eqref{eq:mono-boundary}. Since $u-Kx_1=0$ on $\partial B_R$, its tangential derivative vanish there, and, when $x_1\ge 0$, its outward radial derivative is non-negative by \eqref{eq:mono-basic}. Therefore
\begin{equation*}
    D(u-Kx_1)=\left\langle D(u-Kx_1),\frac{x}{R}\right\rangle\frac{x}{R},\qquad d(x)\coloneqq \left\langle D(u-Kx_1),\frac{x}{R}\right\rangle\geq 0,\qquad x\in \partial B_R\cap \{x_1\ge 0\}.
\end{equation*}
The thesis follows.
\end{proof}
We begin with the first monotonicity results.
\begin{proposition}\label{propositionboh39549695}
   It holds that 
   \begin{equation}\label{eq:mono-angular-positive}
 \omega_i\coloneqq x_i u_1-x_1u_i>0\qquad\text{in }B_R\cap\{x_i>0\},\qquad i=2,\ldots,n
\end{equation}
    and
    \begin{equation}\label{tesirotazioneultimapropriowufnj4834}
   r u_1-x_1u_ r\ge 0\quad \text{in }B_R\cap \{r\ge 0\}.
\end{equation}
\end{proposition}
\begin{proof}
Fix $i=2,\ldots,n$.
A direct computation yields
\begin{equation}\label{eq:mono-angular}
 \Le\omega_i=0\qquad \text{on } B_R\cap \{x_i\ge 0\}.
\end{equation}
Indeed,
\begin{equation*}
\begin{split}
      \Le \omega_i\overset{\eqref{prole4}}&{=}x_i\Le u_1+u_1\Le x_i-x_1\Le u_i-u_i\Le x_1+2\sum_{j,k=1}^n\Ae_{jk}\delta_{ik}u_{1j}-2\sum_{j,k=1}^n\Ae_{jk}\delta_{1k}u_{ij}\\
    \overset{\eqref{prole2},\eqref{prole3}}&{=}4 x_i u_1-8 x_1 x_i\Delta u+2u_1u_i\Delta u-2(D^2 u\,Du)_i u_1-4 x_iu_1 \\
    &\quad -4 x_1 u_i+8 x_1 x_i\Delta u-2u_1u_i\Delta u+2(D^2 u\,Du)_1 u_i+4 x_1 u_i +2\sum_{j=1}^n\Ae_{ji}u_{1j}-2\sum_{j=1}^n\Ae_{j1}u_{ij}\\
    &=0.
\end{split}
\end{equation*}
When $x_i=0$, $\omega_i=0$, because $u$ is even in $x_i$ and thus $u_i$ is odd on $x_i$ . On the other hand, \eqref{eq:mono-boundary} grants 
\begin{equation}\label{auxcontra394959}
    \omega_i=Kx_i\ge 0\qquad\text{on $\partial B_R\cap \{x_i\ge 0\}$}.
\end{equation} 
Therefore, the weak maximum principle \cite[Theorem 3.1]{MR1814364} yields $\omega_i\geq 0$ on $B_R\cap\{x_i>0\}$. If by contradiction there exists $\tilde x\in B_R\cap\{x_i>0\}$ such that $\omega_i(\tilde x)=0$, then by the previous step $\tilde x$ is an interior global minimum of $\omega_i$. Therefore, the strong maximum principle  \cite[Theorem 3.5]{MR1814364} implies that $\omega_i\equiv 0$  on $B_R\cap\{x_i>0\}$, a contradiction with \eqref{auxcontra394959}. Therefore, \eqref{eq:mono-angular-positive} follows. Eventually, \eqref{tesirotazioneultimapropriowufnj4834} is a direct consequence of \eqref{eq:mono-angular-positive} and the invariance of $u$ with respect to $x_2,\ldots,x_n$. 
\end{proof}
With the next result, we provide a sign to $\langle Du,x\rangle$ over the half-ball $B_R\cap\{x_1\ge 0\}$.
\begin{proposition}\label{prop:mono-radial}
It holds that
\begin{equation}\label{segnoxdeu38459694993}
     \langle x,Du\rangle\geq0\qquad\text{in }B_R\cap\{x_1\geq0\}.
\end{equation}
\end{proposition}
\begin{proof}
Notice that, by oddness of $x_1\mapsto u(x_1,\ldots,x_n)$,
\begin{equation*}
    u_j(0,x_2,\ldots, x_n)=0\qquad j=2,\ldots,n.
\end{equation*}
Therefore, $\langle Du,x\rangle=0$ when $\{x_1=0\}$. 
On the other hand, \eqref{eq:mono-boundary} gives $\langle x,Du\rangle=Kx_1+dR\geq0$ on $\partial B_R\cap \{x_1\ge 0\}$. Suppose by contradiction that \eqref{segnoxdeu38459694993} fails. Therefore, there exists $\tilde x=\in B_R\cap \{x_1\ge 0\}$ such that $\langle Du,x\rangle$ has a negative local minimum at $\tilde x$. Assume first that $r>0$ at $\tilde x$. In the forthcoming computations, evaluation at $\tilde x$ is tacitly assumed. Let $\tau_3,\ldots,\tau_n$ be as in \eqref{campiortogonali384956995}.
Define
\begin{equation*}
 e_\varrho\coloneqq\frac{x_1e_1+r e_r}{|x|},
 \qquad e_\theta\coloneqq\frac{r e_1-x_1e_r}{|x|}.
\end{equation*}
Notice that $e_\varrho,e_\theta, \tau_3,\ldots,\tau_n$ is still an orthonormal basis. Moreover,
\begin{equation}\label{segnidiaeb}
     a\coloneqq \left\langle Du,e_\varrho\right\rangle =\frac{\left\langle x,Du\right\rangle}{|x|}<0,
 \qquad
 b\coloneqq\left\langle Du, e_\theta\right\rangle=\frac{r u_1-x_1u_r}{|x|}\overset{\eqref{tesirotazioneultimapropriowufnj4834}}{\ge }0,\qquad |Du|^2=a^2+b^2.
\end{equation}
 A direct computation grants
\begin{equation}\label{eq:mono-radial-ur}
 u_r=\frac{ra-x_1b}{\varrho}<0.
\end{equation}
Since $\tilde x$ is an interior minimum point for $\langle x,Du\rangle$, then $D\langle x,Du\rangle=0$, whence 
\begin{equation*}
    D^2u\,x=-Du,\qquad D^2 u\,e_\varrho=\frac{1}{|x|}D^2 u\,x=-\frac{1}{|x|}Du,\qquad \left\langle D^2u\,e_\varrho,e_\varrho\right\rangle=-\frac {a}{|x|},\qquad  \left\langle D^2u\,e_\varrho,e_\theta\right\rangle=-\frac{b}{|x|}.
\end{equation*}
Combining the above computations with \eqref{eigenvectorsipreliminaries394596} grants
\begin{equation}\label{euqationejekcvoernftrfojnt5849}
    \Delta u=-\frac{a}{|x|}+\langle D^2 u\,e_\theta,e_\theta\rangle+\frac{(n-2)u_r}{r},\qquad \langle D^2 u\,Du,Du\rangle=-\frac{a^3}{|x|}-\frac{2a b^2}{|x|}+b^2\langle D^2 u\,e_\theta,e_\theta\rangle.
\end{equation}
%\[
% \left\langle D^2u\,e_\varrho,e_\varrho\right\rangle=-\frac {a}{|x|},
% \qquad \left\langle D^2u\,e_\varrho,e_\theta\right\rangle=-\frac{b}{|x|}.
%\]
Therefore, \eqref{eq_diri_ball_eps} rewrites as 
\begin{equation}\label{laplacianonellecoordinate39499596}
     \frac{a\bigl(b^2-4|x|^2-\varepsilon^2\bigr)}{|x|}
 +(a^2+4|x|^2+\varepsilon^2)D^2u(e_\theta,e_\theta)
 +(n-2)W_\varepsilon^2\frac{u_r}{r}=4\langle x,Du\rangle.
\end{equation}
In particular, isolating $\langle D^2 u\,e_\theta,e_\theta\rangle$ in \eqref{euqationejekcvoernftrfojnt5849} and substituting it into \eqref{laplacianonellecoordinate39499596} yields
\begin{equation}\label{eq:mono-radial-trace}
 (a^2+4|x|^2+\varepsilon^2)\Delta u
 =4\langle x,Du\rangle-\frac{a(a^2+b^2)}{|x|}-(n-2)b^2\frac{u_r}{r}
 \overset{\eqref{segnidiaeb},\eqref{eq:mono-radial-ur}}{>}4\langle x,Du\rangle.
\end{equation}
Therefore, recalling that $\langle x,Du\rangle<0$ at $\tilde x$, 
\[
 \Delta u>
 \frac{4\langle x,Du\rangle}{a^2+4|x|^2+\varepsilon^2}
 >\frac{4\langle x,Du\rangle}{4|x|^2+\varepsilon^2}.
\]
Consequently, \eqref{prole5} implies
\begin{equation}\label{eq:mono-radial-short}
 \Le\langle x,Du\rangle
 <16\langle x,Du\rangle-(16|x|^2+2\varepsilon^2)
       \frac{4\langle x,Du\rangle}{4|x|^2+\varepsilon^2}
 =\frac{8\varepsilon^2\langle x,Du\rangle}{4|x|^2+\varepsilon^2}<0.
\end{equation}
On the other hand, since $\tilde x$ is a local minimum point of $\langle Du,x\rangle$, then $\Le \langle x,Du\rangle\geq 0$, a contradiction. Finally, we are left to deal with the case $r=0$. In this case, $Du=u_1e_1$ and $\langle x,Du\rangle=x_1u_1<0$, whence $u_1\le 0$. Moreover, since $D\langle x,Du\rangle=0$, then $u_{11}x_1=-u_1$. Thus \eqref{eq_diri_ball_eps} yields
\[
 (u_1^2+4x_1^2+\varepsilon^2)\Delta u
 =4\langle x,Du\rangle+u_1^2u_{11}
 =4\langle x,Du\rangle-\frac{u_1^3}{x_1}>4\langle x,Du\rangle.
\]
Therefore, arguing as in \eqref{eq:mono-radial-short}, we get a contradiction. The thesis follows. 
\end{proof}
Next, we prove monotonicity of $u$ along $e_r$ in the half-ball $B_R\cap \{x_1\ge 0\}$.
\begin{proposition}\label{prop:mono-transverse}
It holds that 
\begin{equation}\label{eq:mono-transverse-positive}
    u_r\geq0\qquad\text{in }B_R\cap\{x_1\geq0\}.
\end{equation}
\end{proposition}
\begin{proof}
Fix $x\in B_R\cap \{x_1\ge 0\}\cap \{r\ge 0\}$. By the discussion before \Cref{propositionboh39549695}, \eqref{eq_diri_ball_eps} rewrites as
\begin{equation}\label{eq:mono-meridional}
 (u_r^2+4|x|^2+\varepsilon^2)u_{11}
   +(u_1^2+4|x|^2+\varepsilon^2)u_{rr}+(n-2)W_\varepsilon^2\frac{u_r}{r}-2u_1u_r u_{1r}
=4\langle x,Du\rangle.
\end{equation}
Since $u(0,x_2,\ldots,x_n)=0$, then $u_r=0$ when $x_1=0$. Moreover, we already know that $u_r=0$ when $r=0$. Finally, on $\partial B_R\cap\{x_1\ge 0\}$, $u_r=d rR^{-1}\geq 0$ by \eqref{eq:mono-boundary}. Assume by contradiction that \eqref{eq:mono-transverse-positive} fails. Then, there exists $\tilde x\in B_R\cap\{x_1\ge 0\}\cap \{r\ge 0\}$ which is a minimum point for $u_r$ and such that $u_r(\tilde x)\le 0$. In particular, $u_{1r}(\tilde x)=u_{rr}(\tilde x)=0$. Moreover, by \eqref{eigenvectorsipreliminaries394596},  
\begin{equation}\label{laplacianincoordinatesdfrfr}
    \Delta u=u_{11}+u_{rr}+(n-2)\frac{u_r}{r}.
\end{equation}
Therefore, \eqref{eq:mono-meridional} implies 
\[
(u_r^2+4|x|^2+\varepsilon^2)\Delta u
 =4\langle x,Du\rangle-(n-2)u_1^2\frac{u_r}{r}\geq0\qquad\text{at }\tilde x,
\]
where the last inequality follows from \eqref{segnoxdeu38459694993} and $u_r(\tilde x)<0$. 
Therefore, since $\tilde x$ is an interior minimum point of $u_r$ and $\Le$ is elliptic,
\begin{equation*}
 0\leq \Le u_r\overset{\eqref{prole7}}{=}4u_r-8r\Delta u+(n-2)W_\varepsilon^2\frac{u_r}{r^2}\le 0\qquad \text{at }\tilde x,
\end{equation*}
a contradiction. The thesis follows.
\end{proof}
As a consequence of \Cref{propositionboh39549695} and  \Cref{prop:mono-transverse}, we gain strict monotonicity of $u$ along $x_1$ in $B_R\cap\{x_1\ge 0\}$ and strict positivity of $\langle Du,x\rangle$ on the same region.
\begin{proposition}\label{prop:mono-positive}
It holds that 
\begin{equation}\label{eq:mono-first-positive}
 u_1>0,\qquad \langle Du,x\rangle\ge 0\qquad\text{in }B_R\cap\{x_1\ge 0\}.
\end{equation}
\end{proposition}
\begin{proof}
Fix $x\in B_R\cap\{x_1\ge 0\}$. Assume first that $r\ge 0$. In this case, the first inequality follows combining \eqref{tesirotazioneultimapropriowufnj4834} with \eqref{eq:mono-transverse-positive}, while the second inequality follows by the first one and by \eqref{eq:mono-transverse-positive}, since $\langle Du,x\rangle =x_1 u_1+r u_r$. Instead, assume that $r=0$, whence $x=(x_1,0,\ldots,0)$. Let $\omega_2$ be as in \Cref{propositionboh39549695}.  Notice that $x$ belongs to the flat portion of the boundary of $B_R\cap \{x_2\ge 0\}$, and the interior unit normal is exactly $e_2$ at $x$. We know by \eqref{eq:mono-angular-positive} that $\omega_2\ge 0$ on $B_R\cap \{x_2\ge 0\}$. Moreover, we saw in the proof of \Cref{propositionboh39549695} that $\omega_2=0$ at $x$. Since $\Le$ is uniformly elliptic, we apply the Hopf Lemma \cite[Lemma 3.4]{MR1814364} to infer
\begin{equation*}
    \partial _{x_2}\big(\omega_2\big)(x)=u_1(x)-x_1 u_{22}(x)\ge 0. 
\end{equation*}
The proof is concluded if $u_{22}(x)\geq 0$. To this aim, by definition of $u_r$ and recalling that $u_2(x)=0$,
\begin{equation*}
    u_{22}(x_1,0,0,\ldots,0)=\lim_{\delta\to 0^+}\frac{u_2(x_1,\delta,0,\ldots,0)}{\delta}=\lim_{\delta\to 0^+}\frac{u_r(x_1,\delta,0,\ldots,0)}{\delta}\overset{\eqref{eq:mono-transverse-positive}}{\geq}0.
\end{equation*}
The thesis follows.
\end{proof}
The proof of \Cref{prooffromboundtoboundderprop3848} is based on the following lower bound. 
\begin{proposition}\label{prop:mono-dilation}
It holds that 
\begin{equation}\label{eq:mono-dilation-conclusion}
 x_1u_1-u\geq0\qquad\text{in }B_R\cap\{x_1\geq0\}.
\end{equation}
\end{proposition}
\begin{proof}
Set $v=x_1u_1-u$. By \eqref{prole1}, \eqref{prole2}, \eqref{prole3} and \eqref{prole4}, 
\begin{equation*}\label{eq:mono-V-full}
\begin{split}
    \Le v&=x_1\Le u_1+u_1\Le x_1+2\sum_{i=1}^n\Ae_{1i}u_{1 i}-8\langle x, Du\rangle +2(4|x|^2+\eps^2)\,\Delta u\\
    &=x_1\big(4 u_1-8x_1\Delta u\big)+u_1\left(2 u_1\Delta u-2\sum_{i=1}^nu_{1 i}u_i-4x_1\right)+2\sum_{i=1}^n\Ae_{1i}u_{1 i}-8\langle x, Du\rangle +2(4|x|^2+\eps^2)\,\Delta u\\
    &=-8x_1^2\Delta u+2 u_1^2\Delta u-4u_1\sum_{i=1}^nu_{1 i}u_i+2\We^2u_{11}-8\langle x, Du\rangle +2(4|x|^2+\eps^2)\,\Delta u\\
    &=2\We^2u_{11}-4u_1\sum_{i=1}^nu_{1 i}u_i-8\langle x, Du\rangle +2\big(u_1^2+4r^2+\eps^2\big)\,\Delta u.
\end{split}
\end{equation*}
Exploiting \eqref{eq_diri_ball_eps},
\begin{equation}\label{serveperplapropdoudn3859}
    \begin{split}
        \Le v&=2\We^2u_{11}-4u_1\sum_{i=1}^nu_{1 i}u_i-2\We^2\,\Delta u+2\sum_{i,j=1}^n u_{ij}u_i u_j +2\big(u_1^2+4r^2+\eps^2\big)\,\Delta u.        .
    \end{split}
\end{equation}
Therefore, when $r\ge 0$ and recalling \eqref{gradienteinroti394969},
\begin{equation}\label{eq:mono-V}
    \begin{split}
        \Le v
        &=2\big(u_r^2+4|x|^2+\eps^2\big)u_{11}-8x_1^2\Delta u-2u_r^2\,\Delta u+2\sum_{i,j=2}^n u_{ij}u_i u_j\\
        &=2\big(u_r^2+4|x|^2+\eps^2\big)u_{11}-8x_1^2\Delta u-2u_r^2\,\Delta u+2u_r^2 u_{rr}\\
        \overset{\eqref{laplacianincoordinatesdfrfr}}&{=}2\big(4|x|^2+\eps^2\big)u_{11}-8x_1^2\Delta u-2(n-2)\frac{u_r^3}{r}
    \end{split}
\end{equation}
Notice that $v=0$ when $x_1=0$. Moreover, on $\partial B_R\cap\{x_1\ge 0\}$, \eqref{eq:mono-boundary} gives $v=dx_1^2R^{-1}\geq 0$. Assume by contradiction that \eqref{eq:mono-dilation-conclusion} fails. Then, there exists $\tilde x\in B_R\cap\{x_1\ge 0\}$ which is a minimum point for $v$ and such that $v(\tilde x)\le 0$. Assume first $r(\tilde x)>0$. Then $v_1=v_r=0$ at $\tilde x$, whence $u_{11}=0$ and $x_1 u_{1r}=u_r$ at $\tilde x$.
The combination of two facts with \eqref{eq:mono-meridional} and \eqref{laplacianincoordinatesdfrfr} grants 
\begin{equation}\label{nonricordopiu38495885}
 (u_1^2+4|x|^2+\varepsilon^2)\Delta u
 =4\langle x,Du\rangle+\frac{2u_1u_r^2}{x_1}-(n-2)\frac{u_r^3}{r} \qquad\text{at }\tilde x,
\end{equation}
and moreover, by \eqref{eq:mono-V},
\begin{equation}\label{auxiliare39459694939393993}
    \Le v=-8x_1^2\Delta u-2(n-2)\frac{u_r^3}{r}\qquad \text{at }\tilde x.
\end{equation}
Therefore, since $\tilde x$ is an interior minimum point of $v$ and $\Le$ is elliptic, \eqref{nonricordopiu38495885} and \eqref{auxiliare39459694939393993} yield
\begin{equation}
 \begin{split}
 \Le v=-
 \frac{8x_1^2}{u_1^2+4|x|^2+\varepsilon^2}
 \left(4\langle x,Du\rangle+\frac{2u_1u_r^2}{x_1}\right)-2(n-2)\frac{u_r^3}{r}
 \left(\frac{u_1^2+4r^2+\eps^2}{u_1^2+4|x|^2+\varepsilon^2}\right)<0,
 \end{split}
\end{equation}
where the last inequality follows combining \eqref{eq:mono-transverse-positive} and \eqref{eq:mono-first-positive}. This is a contradiction.
Finally, assume that $r(\tilde x)=0$.
 In this case, recall that $u_i(\tilde x)=0$ for $i=2,\ldots,n$ by invariance along $x_2,\ldots,x_n$. Moreover, again by minimality of $\tilde x$, $u_{11}(\tilde x)=0$. Therefore \eqref{eq_diri_ball_eps} yields 
\begin{equation}\label{segnodfdfdfdf4578}
 (u_1^2+4x_1^2+\varepsilon^2)\Delta u=4x_1u_1\ge 0\qquad \text{at }\tilde x,
\end{equation}
the last inequality following by \eqref{eq:mono-first-positive}. 
Therefore, \eqref{serveperplapropdoudn3859} and \eqref{segnodfdfdfdf4578} grant  
$$\Le v=-8x_1^2\Delta u<0,$$
again a contradiction. The thesis follows.
\end{proof}
Finally, we prove \Cref{prooffromboundtoboundderprop3848}.
\begin{proof}[Proof of \Cref{prooffromboundtoboundderprop3848}]
   Let $x\in B_R$. Since $x_1\mapsto u_1(x_1,\ldots,x_n)$ is even and continuous, it suffices to assume $x_1\ge 0$.  Combining \eqref{uniformestimatesproprincheigh} and \eqref{eq:mono-dilation-conclusion},
   \begin{equation*}
       x_1 u_1\geq u\geq x_1.
   \end{equation*}
   Since $x_1\ge 0$, the thesis follows dividing by $x_1$.
\end{proof}
\subsection{Proof of \Cref{thm_main_1}}
Let $(R_i)_{i\in \mathbb N_+}$ be any diverging sequence of radii. Set  $(\eps_i)_{i\in \mathbb N_+}\subseteq(0,1]$ be any vanishing sequence of approximating parameters. Let $K_i=K(R_i)$ be as in \eqref{definizionekappaperheightestimates3948568}. Let $u^i\in C^\infty(\overline B_{R_i})$ be the solution to \eqref{eq_diri_ball_eps} given by \Cref{propositiondiriboundcondex}. Fix $j\in\mathbb N_+$. Up to a subsequence, we assume that $B_j\Subset B_{R_i}$ for any $i$. Since $x_1\mapsto u^i(x_1,\ldots,x_n)$ is odd, the combination of  \Cref{quasimaximumprincipleproposition}, \Cref{propunistimamain3948573883derivata}  and \Cref{thm_uniform_gradient} yields a $C^1$-bound for $u^i$ on $\overline B_j$ that is uniform with respect to $i$. Combining Ascoli-Arzelà theorem with a standard diagonal procedure with respect to $j$, there exists $u\in\mathrm{Lip}_\mathrm{loc}(\rr^n)$ such that $u^i\to u$ locally uniformly on $\rr^n$. Define 
\begin{equation*}
    \flui(x,p) =(\flui_1(x,p),\ldots,\flui_n(x,p))\coloneqq \frac{p}{\sqrt{|p|^2+4|x|^2+\eps_i^2}},
\end{equation*}
so that
\begin{equation*}
    \Qei u^i=\divv\flui(x,Du^i).
\end{equation*}
Notice that
\begin{equation*}
    \partial _{x_k}\flui_\ell(x,p)=-\frac{4x_kp_\ell}{\left(|p|^2+4|x|^2+\eps^2\right)^\frac{3}{2}},\qquad  \partial _{p_k}\flui_\ell(x,p)=\frac{\delta_{k\ell}}{\sqrt{|p|^2+4|x|^2+\eps^2}}-\frac{p_kp_\ell}{\left(|p|^2+4|x|^2+\eps^2\right)^\frac{3}{2}}.
\end{equation*}
Therefore, for any fixed $j\in\mathbb N$, $j\geq 2$, $\Qe$ satisfies the assumptions of \cite[Theorem 13.2]{MR1814364} on 
\begin{equation*}
    P_{j+1}\coloneq \left\{\frac{1}{j+1}\le|x|\le j+1\right\},
\end{equation*}
with structural constants that depend on $j$ but that are independent of $i$. Therefore, since we already have a uniform $C^1$-control of $u^i$ on $B_{j+1}$, \cite[Theorem 13.2]{MR1814364} yields a uniform $C^{1,\beta}$-control of $u^i$ on $P_j$. Moreover, we recall that $u_i$ satisfies the linear, second-order, uniformly elliptic equation 
\begin{equation}\label{exlininririririiririr}
\sum_{k,\ell=1}^n\mathcal \mathcal A^{\eps_i}_{k\ell}(x)\big(D^2 u^i\big)_{k\ell}-4 \langle Du^i,x\rangle=0\qquad \text{on $P_j$,}\qquad \Aei(x)\coloneqq A^{\eps_i}(x,Du^i).
\end{equation}
Since $|x|\ge j^{-1}$ on $P_j$, \eqref{eigenvaluesofaeps} and the uniform $C^{1,\beta}$-bounds on $u^i$ yield a uniform ellipticity constant and a uniform $C^{0,\beta}$-bound on the coefficients of \eqref{exlininririririiririr}. Therefore, classical Schauder estimates as in \cite[Theorem 6.2]{MR1814364} grant a uniform $C^{2,\beta}$-bound for $u^i$ on $P_{j-1}$. Therefore, we apply again Ascoli-Arzelà theorem on $P_{j-1}$, concluding that $u\in C^{2,\beta}_{\mathrm{loc}}(\rr^n\setminus\{0\})\cap \mathrm{Lip}_\mathrm{loc}(\rr^n)$ and that $u^i\to u$ locally in $C^2(\rr^n\setminus\{0\})$. In particular, when $x\neq 0$, we let $i\to\infty$ in $\Qei u^i(x)=0$, concluding that 
\begin{equation}\label{equazionenoninzero}
    \divv\left(\frac{Du}{\sqrt{|Du|^2+4|x|^2}}\right)=0\qquad\text{in }\rr^n\setminus\{0\}.
\end{equation}
Away from $0$, one argues \emph{verbatim} as in the proof of \Cref{propositiondiriboundcondex} to show that $u$ solves a linear, second-order, uniformly elliptic equation, and conclude as in that case that $u\in C^\infty(\rr^n\setminus\{0\})$. Fix $x\in\rr^n\setminus\{0\}$. Since, by \Cref{prooffromboundtoboundderprop3848}, $(u^i)_1(x)\geq 1$ for any $i$ sufficiently large, then $C^1$-convergence ensures that $u_1(x)\geq 1$. To conclude, we are left to show that $x$ is smooth at $0$. Indeed, in this case,  \eqref{equazionenoninzero} extends to $0$ by continuity, and $u_1(0)\geq 1$ for the same reason. In order to conclude that $u\in C^\infty(\rr^n),$ we rely on the regularity results of \cite{MR2774306}. Accordingly, for any $i$, define $g^i:\rr^{2n}\to\rr$ and $G^i:\hh^n\to\rr$ by 
\begin{equation*}
    g^i(x,y)\coloneqq\langle x,y\rangle+u^i(x),\qquad  g(x,y)\coloneqq\langle x,y\rangle+u(x)\qquad G^i(x,y,t)\coloneqq g^i(x,y)-t.
\end{equation*}
Since $\Qei u^i=0$,  \eqref{espressionecurvaturamediariempergarficispeciali} grants that, for $i$ sufficiently large,
\begin{equation*}
    \divv\left(\frac{Dg^i+F}{\sqrt{|Dg^i+F|^2+\eps_i^2}}\right)=0\qquad\text{on }B_1,
\end{equation*}
where $F$ is defined in \Cref{sec_tgraph}. Moreover, by the above considerations, $(g^i)_i$ is bounded in $C^1(\overline B_1)$, uniformly with respect to $i$, and converges to $g$ uniformly on $B_1$. Moreover,
\begin{equation*}
    X_1 G^i=y_1+(u^i)_1-y_1=(u^i)_1\geq 1,
\end{equation*}
where the last inequality holds by \Cref{prooffromboundtoboundderprop3848}. Therefore, one argues \emph{verbatim} as in the proof of \cite[Corollary 1.3]{MR2774306} to conclude that $g\in C^\infty(B_1)$. By definition of $g$, $u\in C^\infty(\rr^n)$, and the thesis follows.
\subsection{Proof of \Cref{thm_no_rigidity}}
Let $u$ be as in \Cref{thm_main_1}. Define $S\subseteq\hh^n$ by
\begin{equation*}\label{il_controesempio}
        S=\{(x,y,t) \in \hh^n\, :\,t=\left\langle x,y\right\rangle+u(x)\}.
    \end{equation*}
   Clearly, $S$ is smooth, connected, embedded, two-sided, complete, without boundary. Since $u_1\geq 1$, \eqref{carsetforspecialgraf} implies that $S$ is non-characteristic. Moreover, by \eqref{espressionecurvaturamediariempergarficispeciali}, \Cref{thm_main_1} grants that $S$ is minimal. Since $S$ is a $t$-graphs, \cite[Example 2.6]{MR2333095} (cf. \cite[Theorem 2.3]{MR2333095}) implies that $S$ is the boundary of a global perimeter minimizer. In particular, $S$ is stable, and satisfies \eqref{updesubintro} e.g. by \cite[Theorem 2.2]{psv}. In addition, notice that $S=\{G=0\}$, where $G(x,y,t)=\langle x,y\rangle+u(x)-t$. Since $u_1\geq 1$, then $X_1 G= u_1\geq 1$. Therefore, by \Cref{propcriteriointrinsic}, $S$ is an entire, smooth intrinsic graph.  Since $S$ is a $t$-graph, it is not a vertical hyperplane.
Finally, we show that $S$ does not satisfy \eqref{eq:P1}, \eqref{eq:P2} and \eqref{eq:P3}.
Since $u$ is odd in $x_1$, then
$u=0$ on $\{x_1=0\}$, and $u_i=0$ on $\{x_1=0\}$ for any $i=2,\ldots,n$. Define the curve $z:\rr\to\rr^{2n}$ by $z(r)=(x(r),0)\in\rr^{2n}$, with $x(r)=(0,r,0\ldots,0)$. The quantities below are evaluated at $z(r)$. By \Cref{sec_specialclass}, 
\begin{equation}\label{componierfnerivjnvijvn}
      \nu=-\frac{u_1X_1+2rY_2}{W},
    \qquad
    J(\nu)=\frac{2rX_2-u_1Y_1}{W},
    \qquad
    \alpha=\frac1W,
    \qquad
    W^2=u_1^2+4r^2.
\end{equation}
Define
\[
    \mathbf B:=\frac{2rX_1-u_1Y_2}{W},
    \qquad
    J(\mathbf B)=\frac{u_1X_2+2rY_1}{W}.
\]
Both vector fields have unitary norm. By \eqref{componierfnerivjnvijvn}, both belongs to $\hhh' TS$ over $z(r)$. Denote by $\frac{d}{dr}$ differentiation along $z(r)$. Since
\[
    \frac{dW}{dr}=\frac{u_1u_{12}+4r}{W},
\]
then
\begin{equation*}
    \nabla_{J(\nu)}\nu
    =-\frac{2r}{W}
      \left[
        \frac{d}{dr}\left(\frac{u_1}{W}\right)X_1
        +\frac{d}{dr}\left(\frac{2r}{W}\right)Y_2
      \right]=\frac{4r(u_1-ru_{12})}{W^3}\,\mathbf B.
\end{equation*}
Therefore,
\begin{equation*}
     h^{\hhh}(J(\nu),\mathbf B)
   =\frac{4r(u_1-ru_{12})}{W^3}.
\end{equation*}
On the other hand, notice that, since $u_1\geq 1$, 
\[
   \lim_{r\to 0} u_1(z(r))-ru_{12}(z(r))= u_1(0)\geqslant1, 
\]
so that $ h^{\hhh}(J(\nu),\mathbf B)\neq 0$ 
sufficiently close to (but away from) $0$. Therefore, \eqref{eq:P1} fails.
In addition,
\begin{align}
    J(\mathbf B)\alpha
    &=\frac{u_1}{W}\frac{d}{dr}\left(\frac1W\right),\label{erjcnefivjntrvijnvi}\\
    J(\nu)\alpha+\alpha^2
    &=\frac{2r}{W}\frac{d}{dr}\left(\frac1W\right)+\frac1{W^2}
      =\frac{d}{dr}\left(\frac{r}{W^2}\right)\label{erjfnerifjntrfijtnf}.
\end{align}
On the other hand, since $W^2=u_1^2+4r^2\geqslant4r^2$, then
\begin{equation}\label{erinrfirjtnvrifcjnerijrn}
    0<\frac1W\leqslant\frac1{2r}\longrightarrow0,
    \qquad
    0<\frac{r}{W^2}\leqslant\frac1{4r}\longrightarrow0
    \qquad\text{as }r\to+\infty.
\end{equation}
Assume by contradiction that \eqref{eq:P2} holds. Since $u_1\geq 1$, \eqref{erjcnefivjntrvijnvi} implies that $W$ is constant over $z(r)$, a contradiction with \eqref{erinrfirjtnvrifcjnerijrn}. Then \eqref{eq:P2} fails. Finally, assume by contradiction that \eqref{eq:P3} holds. Then, by \eqref{erjfnerifjntrfijtnf}, $r/W^2$ is
non-decreasing on $(0,+\infty)$, contradicting \eqref{erinrfirjtnvrifcjnerijrn}. Thus \eqref{eq:P3} fails, and the thesis follows.

\section*{Declarations}
\footnotesize{
%\noindent \textbf{AI usage statement} 
%The authors acknowledge the use ChatGPT-6 Astra, in the following way. 
%Initially, the authors asked ChatGPT-6 Astra to look for relations occurring between minimality and \eqref{eq:P1}. Precisely, the authors gave \cite{simons} as input, asking ChatGPT-6 Astra to focus on \cite[Example 3.8]{simons}, namely \eqref{eofjnt4rvjntrgvr} of this paper, as an instance of characteristic minimal hypersurface satisfying \eqref{eq:P1}, and asked it try to understand how a perturbation as in \eqref{eofjnt4rvjntrgvr} would have affected either minimality or \eqref{eq:P1}. After a long series of interactions, ChatGPT-6 Astra noticed that a solution to \eqref{eq:target}, because of its derivative bound, would have produced a minimal non-characteristic hypersurface which would have not satisfied \eqref{eq:P1}. Incidentally, the authors realized that such a construction would have provided a counterexample to the higher dimensional Bernstein problem. Subsequently, the authors implemented the Riemannian approximation scheme. To this aim, the contribution of ChatGPT-6 Astra have been clearly crucial in detecting the barriers in \Cref{erfnjtrvijng}. Moreover, ChatGPT-6 Astra helped in sketching the monotonicity argument in \Cref{fvrtg5ytbgtygvtr}. 
%The paper has been entirely written and checked by the authors. The authors take full
%responsibility for the content of the paper.
%\smallskip
%
\noindent{\textbf{Conflict of interest.}} The authors have no financial interests or conflicts of interest related to the subject matter.
\smallskip

\noindent \textbf{Data availability statement.} Data sharing not applicable as no datasets were generated or analyzed during this study.}
\bibliographystyle{abbrv}
\bibliography{biblio}
\end{document}